\documentclass[12pt,twoside]{amsart}

\usepackage[latin1]  {inputenc}%
\usepackage[T1]      {fontenc }%
\usepackage          {amsmath }%
\usepackage          {amsfonts}%
\usepackage          {amssymb }%
\usepackage          {amsthm  }%
\usepackage          {a4wide  }%
\usepackage          {url     }%
\usepackage          {tikz    }%
\usepackage[bookmarks=false,pdfborder={0 0 0.05}]{hyperref}
\usepackage[all]{xy}

\usepackage{enumerate, amsmath, amsfonts, amssymb, amsthm,  wasysym, graphics, graphicx, xcolor, frcursive,comment,bbm}

\usepackage{etex}
\usepackage{MnSymbol}

\definecolor{darkblue}{rgb}{0.0,0,0.7} 
\definecolor{darkred}{rgb}{0.7,0,0} 

\usepackage{hyperref}
\usepackage[all]{xy}
\usepackage[T1]{fontenc}

\usepackage{array}

\usepackage{float}
\restylefloat{table}

\def\defn#1{{\sf #1}}

\newcommand{\RR}{\mathbb R}
\newcommand{\ZZ}{\mathbb Z}

\newcommand{\spanz}{\mathrm{span}_{\mathbb{Z}}}
\newcommand{\spanr}{\mathrm{span}_{\mathbb{R}}}

\DeclareMathOperator{\Red}{Red}

\DeclareMathOperator{\idop}{id}

\DeclareMathOperator{\fin}{fin}

\DeclareMathOperator{\Redd(c)}{Red_T(c)^{gen}}

\DeclareMathOperator{\GL}{GL}
\DeclareMathOperator{\TR}{tr}

\DeclareMathOperator{\END}{End}

\DeclareMathOperator{\Z}{\mathbb{Z}}

\DeclareMathOperator{\NN}{\mathbb{N}}

\DeclareMathOperator{\XX}{\mathbb{X}}

\DeclareMathOperator{\COH}{coh}

\DeclareMathOperator{\ARS}{\Phi_{aff}}

\DeclareMathOperator{\FRS}{\Phi_{fin}}

\DeclareMathOperator{\VFIN}{V_{fin}}
\DeclareMathOperator{\VAFF}{V_{aff}}

\DeclareMathOperator{\Thi}{Thick}
\DeclareMathOperator{\cox}{cox}

\DeclareMathOperator{\aff}{aff}

\newtheorem{theorem}{Theorem}[section]
\newtheorem{corollary}[theorem]{Corollary}
\newtheorem{Proposition}[theorem]{Proposition}
\newtheorem{Lemma}[theorem]{Lemma}

\theoremstyle{definition}
\newtheorem{Definition}[theorem]{Definition}

\newtheorem{lemdef}[theorem]{Lemma and Definition}
\newtheorem{remark}[theorem]{Remark}
\newtheorem{hyp}[theorem]{Hypothesis}

\newtheorem*{remarkOHNE}{Remark}

\AtBeginDocument{%
   \def\MR#1{}
}

\title{The hyperbolic cover of an elliptic Weyl group}

\author[B.~Baumeister]{Barbara Baumeister}
\address{Barbara Baumeister, Universit\"at Bielefeld, Germany}
\email{b.baumeister@math.uni-bielefeld.de}
\thanks{}
\author[P.~Wegener]{Patrick Wegener}
\address{Patrick Wegener}
\email{pat-weg09@gmx.de}


\date{\today}
\begin{document}

\newcolumntype{C}[1]{>{\centering\arraybackslash}m{#1}}

\begin{abstract}
In this paper, we study in detail the hyperbolic covers $\tilde{W}$  and $\hat{W}$ of an elliptic Weyl system introduced by Saito \cite{Sai85}. We show that they are isomorphic and also isomorphic to an extended Coxeter system of star type. For $\tilde{c}$ a Coxeter transformation in $\tilde{W}$ we can conclude the Hurwitz transitivity of the braid group action on the set of reduced reflection factorizations of $\tilde{c}$  from the Hurwitz transitivity in extended Coxeter systems of star type, which is shown in \cite{BWY24}. This then enables us to establish for a weighted projective line $\XX$ of tubular type \cite{GL87, Len96} an order preserving bijection between the poset of thick subcategories of $\COH(\mathbb{X})$ generated by an exceptional sequence and the poset $[\idop, \tilde{c}]$ ordered by the absolute order. In an Appendix, we study the hyperbolic cover of a Coxeter system.
\end{abstract}
\maketitle

\setcounter{tocdepth}{1}
\tableofcontents

\section{Introduction}\label{Sec:Intro2}

This note is part of the  systematic study of \defn{extended Weyl groups} $W$ and \defn{extended Weyl systems} (see \cite{BWY24, BWY21, BMN24}). An extended Weyl group is a reflection group, that is  a group generated by reflections of a finite dimensional  $\RR$-space $V$, which is equipped with a bilinear form (see \cite{BWY24}). There are three different types of extended Weyl groups, those of \defn{domestic}, \defn{tubular} and of \defn{wild} type corresponding to the three possible types of the related bilinear form. The extended Weyl groups of domestic type are precisely the simply-laced affine Coxeter groups (see Remark~2.5 (c) of \cite{BWY24}), and those of tubular type are \defn{ elliptic Weyl groups} (see \cite{BWY21, KL86, Sai85}). 

Extended Weyl groups share several properties with Coxeter groups. For instance, for every extended Weyl group the vector space $V$ contains a set of `roots' that form a  generalized root system, and which contains a basis consisting of roots, such that the set $S$ of reflections with respect to this set of roots is a  `simple system` for $W$ (see Section~\ref{Subsec:EllipticWeylgroup}). 

The motivation of studying extended Weyl groups is manifold. One is the interest in simple elliptic or hyperbolic singularities whose monodromy groups are extended Weyl groups (for an overview, see \cite{Ebe19}, or also \cite{Sai85}). 

Another interest is in representation theory of algebras for hereditary categories, whose related Weyl groups turn out to be extended Weyl groups (see \cite{GL87, Len96, STW16, BWY21}). Yet another motivation is the interest in the Artin groups of extended type and the $K(\pi,1)$-conjecture (see \cite{Del72, ChaDa94, PS19, Lek, Yamada, BMN24}).

For these questions the Coxeter transformation $c \in W$, which is an analog of a Coxeter element in a Coxeter group, and  the set of reflections $T = \cup_{w \in W} w^{-1}Sw \subseteq W$  in $W$ (see Section~\ref{Sec:Reflections})  play a special role as well as an action of the braid group on the set of reduced reflection factorizations Red$_T(c)$, the so called Hurwitz action (or dual Matsumoto property).
 The \defn{Hurwitz action} of the braid group $\mathcal{B}_n$ on $n$ strands on $T^n$  is defined by
\begin{align*}
	\sigma_i (t_1 ,\ldots , t_n )      & = (t_1 ,\ldots , t_{i-1} , \hspace*{5pt} \phantom
	{t_i}t_{i+1}\phantom{g_i^{-1}}, \hspace*{5pt} t_{i+1}^{-1}t_it_{i+1}, \hspace*{5pt} t_{i+2} ,
	\ldots , t_n),\\
	\sigma_i^{-1} (t_1 ,\ldots , t_n ) & = (t_1 ,\ldots , t_{i-1} , \hspace*{5pt} t_i t_{i+1} t_i^{-1},
	\hspace*{5pt} \phantom{t_{i+1}}t_i\phantom{t_{i+1}^{-1}}, \hspace*{5pt} t_{i+2} ,
	\ldots , t_n).
\end{align*}
for the standard generator $\sigma_{i}\in \mathcal{B}_n$, for $1\leq i \leq n-1$ and $(t_{1},\ldots,t_{n})$ an element of $T^{n}$. 

 For the mentioned applications, the Hurwitz action is needed to be transitive on Red$_T(c)$ (see \cite{Ebe19, HK16, BWY21, Bes03}).
 If $W$ is of domestic or wild type, then the action is transitive (see \cite{ BWY24}), while this is not true for the elliptic Weyl systems (see \cite{BWY21}). There, the Hurwitz action is transitive on the proper subset $\Redd(c)$ of Red$_T(c)$ and not on Red$_T(c)$; and this is even not true for the elliptic Weyl group of type $D_4^{(1,1)}$, as  Charly Schwabe later oberserved (see \cite{BWY21})\footnote{Transitivity had mistakenly been claimed  for $D_4^{(1,1)}$ in a previous version of \cite{BWY21}.}

This paper  presents  a solution to this obstruction and thereby  makes possible  a uniform proof of the Hurwitz transitivity. The idea is to move from the elliptic Weyl system 
to its hyperbolic covers as has been proposed to us by Kyoji Saito.
Having an extended Weyl system that is not elliptic
we proceeded as follows.
 In \cite[Section~5]{BWY24} we showed that every 
 extended Weyl system of domestic or wild type
 is an extended Coxeter system. Extended Coxeter groups have been introduced by Looijenga \cite{Looi80}, and have also been studied 
  by van der Lek \cite{Lek}, as groups of affine transformations (see Section~\ref{Sec:ExtCox}). An extended Coxeter group is a split extension of a crystallographic Coxeter group by its root lattice. The Coxeter groups that appear in 
the extended Weyl groups of domestic or wild type  have a Dynkin diagram, which is a star \cite{BWY24} (for the notion of a star see Subsection~\ref{Sec:Notation}). Therefore, we say that
the related extended Coxeter systems are of \defn{star type}.
 For those groups it is shown uniformly:

  \begin{theorem}\label{Thm:11in 4}\cite[Theorem~8.6]{BWY24}
Let $(\mathcal{W}, \mathcal{S})$ be an extended  Coxeter 
system  of star type with set of reflections $\mathcal{T}$ and Coxeter transformation $c$. Then the Hurwitz action is transitive on 
the set of reduced reflection factorizations $\Red_{\mathcal{T}}(c)$. 
  \end{theorem}
  
Here we focus on elliptic Weyl systems,  or more precisely, mostly on the important subclass of the elliptic Weyl systems whose diagrams only have simple bonds and which are of codimension 1 (see \cite[Section 12]{Sai85},\cite{Sai90, Yamada}). They  are those that are relevant for our applications. We  call them \defn{tubular elliptic Weyl systems} (see Section~\ref{Subsec:EllipticWeylgroup}).

We pursue the following strategy.
First, we  recall the elliptic Weyl systems (Section~\ref{sec:Notation}).
Then we will discuss  in detail the hyperbolic covers of an elliptic Weyl system (they appeared under the name hyperbolic extension in \cite{Sai85, ST97}). Also, for further use, we carefully construct the covers $(\tilde{W}, \tilde{S})$ and $(\hat{W}, \hat{S})$ (Sections~\ref{Sec:HyperbolicExtended} and \ref{sec:HypCovers}), and show that they are isomorphic and a central non-split extension of $W$ (Proposition~\ref{StructureTildeW}). We prove that the (unique up to isomorphism) cover is not a Coxeter system for any simple system (Theorem~\ref{Thm:NotCoxeter}). 

As the hyperbolic cover lives in a space with a non-degenerate bilinear form we can easily determine the reflection
length of a Coxeter transformation in $\tilde{W}$, and  thereby we get a new proof for the reflection length $\ell_T(c)$ of a Coxeter transformation $c$ in $W$.

\begin{Proposition}\label{Lem:redCoxElt}
	Let $(W,S)$ be a tubular elliptic Weyl system and $(\tilde{W}, \tilde{S})$ be its hyperbolic cover. Then  $ \ell_{\tilde{T}}(\tilde{c})$ equals $\ell_T(c) $, and $\ell_T(c) $  equals the rank $|S|$ of the elliptic Weyl system.
\end{Proposition}

We will show that the
 hyperbolic cover of a tubular elliptic Weyl system is an extended Coxeter system of star type, and also that every extended Coxeter system of star type with affine Coxeter subgroup is the hyperbolic cover of a tubular elliptic
 Weyl system (Theorem~\ref{Prop:CharExtendedWeylNotTubular}). 
From this it follows that the extended Coxeter systems of star type and the hyperbolic covers
of extended Weyl groups describe the same objects (see \cite[Theorem~5.3]{BWY24}).
Also, as a consequence of Theorem~\ref{Prop:CharExtendedWeylNotTubular} Hurwitz transitivity concludes from
 Theorem~\ref{Thm:11in 4}.

 \begin{theorem} \label{thm:MainElliptic}
	Let $\Phi$ be a tubular elliptic root system, $B =B(\Phi)$ a tubular elliptic root basis, and $\tilde{c}$ a Coxeter transformation with respect to $B$ in the hyperbolic cover $ \tilde{W}$. Then the Hurwitz action is transitive on the set of reduced reflection factorizations $\Red_{\tilde{T}}(\tilde{c})$ of $\tilde{c}$ in $\tilde{W}$.
\end{theorem}

We also obtain the Hurwitz orbits of the braid group on $\Red_T(c)$ 
(Corollary~\ref{Cor:HurwitzW-tildeW}).

As an application of Proposition~\ref{Lem:redCoxElt} and Theorem~\ref{thm:MainElliptic} we prove the following theorem, whose notation can be found in the last section of \cite{BWY21}. It presents the new uniform approach  of an isomorphism of posets between the poset of subobjects of a hereditary category and a
combinatorial poset $[\idop,c] \subset W$ of the related extended Weyl group (see Section~\ref{Sec:IntervalPosets}) formulated for a
weighted projective line of tubular type, which is extending results of \cite{HK16}.

\begin{theorem} \label{conj:WeightProjElliptic}
	Let $\XX$ be a weighted projective line of tubular type over an algebraically closed field of characteristic zero with corresponding elliptic Weyl group $W$,  and hyperbolic cover $\tilde{W}$ of $W$, with set of reflections $\tilde{T}$, and Coxeter transformation 
	 $\tilde{c} \in \tilde{W}$. Then there exists an order preserving bijection between
	\begin{itemize}
		\item the poset of thick subcategories of $\COH(\XX)$ that are generated by an exceptional sequence, ordered by inclusion; 
		and 
		\item the poset $[\idop, \tilde{c}]$, ordered by the absolute order.
	\end{itemize}
\end{theorem}
\noindent

In an appendix to this paper, we consider the Coxeter systems of finite rank, and construct for them the hyperbolic covers.  We show that no new structures are appearing, more precisely that every hyperbolic cover of a Coxeter system is isomorphic to that Coxeter system (Proposition~\ref{Thm:HyperCoverCoxeter}).
In particular, this provides a simple argument for the fact that every Coxeter system has a faithful representation in a non-degenerate space. We use this knowledge in \cite{BWY24}. There, the bijection given in  \cite{HK16} is generalized to a uniform combinatorial description of all  hereditary connected ext-finite abelian $k$-categories with a tilting object over an algebraic closed field in characteristic zero.
\bigskip
\\
\textbf{Acknowledgement.} The authors would like to thank Ky\={o}ji Sait\={o} for  his advice to study the Hurwitz action in the hyperbolic covers. We also like to thank Jon Mccammond and
Georges Neaime for various discussions on the  hyperbolic
cover of the system of type $D_4^{(1,1)}$.
Last not least, we thank an anonymous referee for various comments on 
a previous version of the paper, and Charly Schwabe for pointing out a serious misprint.
The first author acknowledges that this work was supported by the Deutsche Forschungsgemeinschaft (SFB-TRR 358/1 2023 - 491392403).

\subsection{Notation}\label{Sec:Notation}
Let $V$ be an $\RR$-space, and $f,g \in \mathrm{End}(V)$.  We denote by $fg$ the endomorphism $fg : V \rightarrow V$, where $(fg)(v) := f(g(v))$ for all $v \in V$.

Further, if $G$ is a group, $g_1, \ldots, g_r \in G$ and $i \in \{1, \ldots , r\}$, then $g_1 \cdots \widehat{g_i} \cdots g_r$ denotes the product $g_1 \cdots g_{i-1} g_{i+1} \cdots g_r$, where we omit $g_i$.

We call a graph a \defn{star}, if it is  a tree and if there is at most one vertex of degree larger than $2$ and all the other vertices have degree $1$ or $2$.

\section{Elliptic root systems and elliptic Weyl groups}\label{sec:Notation}

In this section, we introduce the notation that we use throughout this paper. In particular, we recall the definition of a generalized root system and define the extended Weyl groups (see also \cite{Sai85} or \cite{STW16}).

\subsection{Elliptic  root systems and elliptic spaces }\label{subsec:notation}
Let $V$ be a finite dimensional $\RR$-space and $(-\mid -)$ a symmetric bilinear form on $V$. A vector $\alpha$ is \defn{non-isotropic}, if it holds $(\alpha \mid \alpha)\neq 0$, and the \defn{reflection} of $V$ with respect to the non-isotropic $\alpha$ is 
$$s_\alpha(v):= v -  \frac{2(\alpha\mid v)}{(\alpha \mid \alpha)} \alpha,~\mbox{for all}~v \in V.$$

The following definition is a generalization of the notion of a root system in the theory of finite Coxeter systems (see  \cite[Section 1.2]{Sai85}).

\begin{Definition}\label{def:root_system}
A non-empty subset $\Phi\subseteq V$ of non-isotropic vectors is called \defn{generalized root system} if the following properties are satisfied
\begin{enumerate}
	\item[(a)] $\spanr(\Phi)= V$,
	\item[(b)] $s_{\alpha}(\Phi) \subseteq \Phi$ for all $\alpha\in \Phi$ and
	\item[(c)] $\Phi$ is \defn{crystallographic}, i.e. $\frac{2(\alpha \mid \beta)}{(\beta \mid \beta)} \in \mathbb{Z}$ for all $\alpha,\beta \in \Phi$.
\end{enumerate}
	
We say that $\Phi$ is of \defn{rank} $m$, if $\dim V = m$. The elements in a generalized root system are \defn{roots}. The generalized root system  $\Phi$  is  \defn{reduced} if $r\alpha\in \Phi$ ($r\in \RR$) implies $r=\pm 1$ for all $\alpha\in \Phi$, and the system is \defn{irreducible} if there do not exist generalized root systems $\Phi_{1},\Phi_{2}$ such that $\Phi=\Phi_{1} \cup \Phi_{2}$ and $\Phi_{1} \bot \Phi_{2}$. A non-empty subset $\Psi$ of a generalized root system $\Phi$ is a \defn{root subsystem} if $\Psi$ is a generalized root system in $\spanr(\Psi)$. For a subset $\Psi=\{\alpha_1, \ldots, \alpha_n \} \subseteq \Phi$ we define $\langle \Psi \rangle_{\text{RS}}$ to be the smallest root subsystem of $\Phi$ that contains  $\Psi$, and we call it the \defn{root subsystem generated by $\Psi$}.
	
For a subset $\Psi \subseteq \Phi$ we put $L(\Psi) = \spanz(\Psi)$. The crystallographic condition on  $\Phi$ implies  that $L(\Psi)$ is a lattice.  The lattice $L(\Phi)$ is the \defn{root lattice of $\Phi$}. The root system $\Phi$  is  \defn{simply-laced} if $(\alpha \mid \alpha) =2$ for all $\alpha \in \Phi$. Two generalized root systems $\Phi_{1}$ and $\Phi_{2}$ are  \defn{isomorphic} if there exists a linear isometry between the corresponding ambient spaces that sends $\Phi_{1}$ to $\Phi_{2}$.
\end{Definition}

\begin{Definition}\label{def:elliptic}
We call a generalized root system $\Phi$ \defn{elliptic} root system if 
\begin{itemize}
\item[(a)] $\Phi$ is irreducible and reduced
\item[(b)] The signature of $(-\mid -)$ is $(n,0,2)$, where the first, second and third entries are the number of positive, negative, and zero eigenvalues of the form and where $n +2 = \dim V$.
\end{itemize}
If $\Phi$ is an elliptic root system, then $(V,(-\mid-))$ is called 
\defn{elliptic} space.
\end{Definition}

In this note, we denote the radical of the bilinear form $(-\mid -)$ by  $R$. Moreover, we abbreviate $(\Phi+X)/X = \{ \alpha + X\mid \alpha \in \Phi\}$  by $\Phi/X$ for $X$ a subspace of $R$.
In addition, we denote by $p_X$  the canonical epimorphism $V \rightarrow V/X$ and by $(- \mid -)_X$  the induced bilinear form on $V/X$, that is, $(p_X(v) \mid p_X(w))_X = (v \mid w)$ for all $v,w \in V$. 

\begin{Lemma}\label{Lem:QuotientsEllipticRootS}
Let $\Phi$ be an elliptic root system and $U$ a  subspace of the radical $R$. Then the following hold.
\begin{itemize}
\item[(a)] $\Phi/U$ is an irreducible generalized root system.
\item[(b)] $\Phi/R \subseteq V/R$ is a finite irreducible  root system and $(-\mid -)$ induces a positive definite form on $V/R$.
\item[(c)] If $U$ is $1$-dimensional, then $\Phi/U \subseteq V/U$ is an affine irreducible root system, and \\$(-\mid-)$ induces a semi-positive definite form on $V/U$, whose radical is $1$-dimensional.
\end{itemize}
\end{Lemma}

\begin{proof}
Suppose that $\Phi/U  = (\Psi_1+U)/U \cup (\Psi_2 +U)/U$ and $(\Psi_1 + U)/U \bot (\Psi_2 +U)/U$. Then $(\Psi_1+U) \cup (\Psi_2 +U) = \Phi+U$, and therefore $\Phi_1 \cup \Phi_2 = \Phi$, where $\Phi_i:= (\Psi_i + U) \cap \Phi$ for $i = 1,2$. As $U \leq R$, the assumption $(\Psi_1 +U)/U \bot (\Psi_2+U)/U$ yields  $\Phi_1 \bot \Phi_2$.  We conclude that $\Phi_i \subset \spanr (\Phi_i)$ are generalized root systems for $i = 1,2$ contradicting the fact that $\Phi$ is irreducible. This proves (a).
	
Let $A$ be the Gram matrix of $(-\mid -)$. As $\Phi$ is irreducible by our assumption, the matrix $A$ is indecomposable. Therefore, (b) and (c) follow from (a) and \cite[Lemma~4.5]{Kac83}.
\end{proof}

 From now on, we assume the following setting in this paper.

 \begin{hyp}\label{Hyp}
~~~
\begin{itemize}
\item[(H1)] $\Phi$ is an elliptic root system,
\item[(H2)] $(\alpha \mid \alpha ) = 2$ for all $\alpha \in \Phi$;
\item[(H3)] It is $U \leq R$  a fixed $1$-dimensional subspace of the radical $R$ such that $U \cap L(\Phi) \neq \{0\}$.
\end{itemize}
\end{hyp}

\begin{remark}
We claim that such a subspace as required in (H3) exists. 
By Lemma~2.3 (c) the root system is not finite (see \cite[Proposition~4.9]{Kac83}). This yields $R \cap  L(\Phi)  \neq \{0\}$, since otherwise the finite system $\Phi/R$ would be isomorphic to $\Phi$.
\end{remark}

Next, we describe $\Phi$ (see Proposition~\ref{RootSystem}). Let $\{\beta_0, \ldots , \beta_n\}$ be a root basis of  the affine root system $\Phi/U$ such that $\{\beta_1, \ldots , \beta_n\}$ is a root basis of a finite root system. By \cite[Theorem 5.6]{Kac83} there exist unique $m_1, \ldots, m_n \in \NN$ such that 
$$\beta_0 + \sum_{i = 1}^n m_i \beta_i~\mbox{}$$
spans the lattice $L(\Phi/U) \cap R/U$. Let $\alpha_i$ be a preimage of $\beta_i$ in $\Phi$ for $0 \leq i \leq n$ and set
$$b = \alpha_0 + \sum_{i = 1}^n m_i \alpha_i.$$
Then $b$ is in $L(\Phi) \cap R$ and 
$$\spanz(b +U) = L(\Phi/U) \cap R/U.$$
Let $a$ in $U$ such that $L(\Phi) \cap U = \ZZ a$. Further set 
$$V_b:= \VAFF := \sum_{i = 0}^n \RR \alpha_i~\mbox{and}~\VFIN := \sum_{i = 1}^n \RR \alpha_i.$$

\begin{Lemma}\label{Lem:Types}
The following holds.
\begin{itemize}
\item[(a)] $\FRS := \VFIN \cap \Phi$ is a finite root system with root basis $B_{\fin}:= \{\alpha_1 , \ldots , \alpha_n\}$.
\item[(b)]  $\Phi_b:= \ARS  := \VAFF \cap \Phi$ is an affine root system with root basis $B_b:= B_{\aff}:= \{\alpha_0, \alpha_1 , \ldots , \alpha_n\}$.
\end{itemize}
\end{Lemma}

\begin{proof} 
By construction $(- \mid -)$ restricted to $\VFIN$ (resp. to $ \VAFF$) is positive definite (resp. semi-positive with $1$-dimensional radical), see  Lemma~\ref{Lem:QuotientsEllipticRootS}, which yields the assertion by \cite[Lemma~4.5]{Kac83} .
\end{proof}

\begin{Lemma}\label{RadicalLattice}
It is $L(\Phi) \cap R = \ZZ a \oplus \ZZ b$. 
In partiuclar, $L(\Phi) \cap R$  is a full lattice in $R$.
\end{Lemma}
\begin{proof}
By the choice of $U$, it is $a \neq 0$. Further the choice of $a$ and $b$ yields $\ZZ a \cap \ZZ b = \{0\}$ and $a, b \in  L(\Phi) \cap R$. Therefore, $\ZZ a \oplus \ZZ b$ is a  subset of $L(\Phi) \cap R$.
	
Now, let $r \in L(\Phi) \cap R$. Then $p_U(r) \in L(\Phi)/U \cap R/U$. The definition of $b$ implies that $p_U(r) = z b +U$  for some $z \in \ZZ$. As $r, b \in L(\Phi)$ we obtain $r - z b \in L(\Phi) \cap U = \ZZ a$, which yields the assertion.
\end{proof}

\begin{Proposition}\label{RootSystem}
Let $\Phi$ be an elliptic root system with respect to $(- \mid -)$ of rank at least four. Then $ \displaystyle \Phi = \FRS \oplus \ZZ a \oplus \ZZ b$.
\end{Proposition}

\begin{proof} 
First we show $\Phi \subseteq  \Phi_{\fin} \oplus \ZZ a \oplus \ZZ b$. Let $\alpha \in \Phi$. Then we get $p_U(\alpha) \in \ARS$ and
$$p_U(\alpha) = \sum_{i = 1}^n z_i \beta_i + z p_U(b)~\mbox{for some}~z_i,z \in \ZZ.$$
Therefore,  $\alpha = \sum_{i = 1}^n z_i \alpha_i + z b + ra$ for some $r \in \RR$ and  $ra = \alpha - \sum_{i = 1}^n z_i \alpha_i - z b \in L(\Phi) \cap U = \ZZ a$. Thus $r \in \ZZ$, which yields, as 
$$\sum_{i=1}^n z_i \alpha_i \in \Phi_{\fin},$$
the claimed inclusion.
	
For the other inclusion, consider for $\alpha \in  \Phi_{\fin}$ the set $K_{R}(\alpha):= \{x \in R \mid  \alpha + x \in \Phi\},$ which is studied in \cite[1.16]{Sai85}. As $\Phi$ is simply-laced  and $\FRS$ irreducible by Lemma~\ref{Lem:QuotientsEllipticRootS} (b), all roots in $\Phi_{\fin}$ are conjugate, which implies that  
\medskip\\
(1.) $K_{R}(\alpha) = K_{R}(\beta)$  for every pair of roots $\alpha, \beta$ in $ \Phi_{\fin}$.
\medskip\\
We show that $K_{R}(\alpha) $ is a sublattice of $L(\Phi)$. Let $x, y \in K_{R}(\alpha) $. As $\Phi_{\fin}$ is irreducible and of rank at least $2$, we can choose $\alpha_i, \alpha_j \in \FRS$  such that  $(\alpha_i \mid \alpha_j) = -1$. Then $\alpha_i + x, \alpha_j +y \in \Phi$, which yields $s_{\alpha_j +y}(\alpha_i + x) = \alpha_i +x + \alpha_j +y \in \Phi$, as well.  We conclude that  $\alpha_i + \alpha_j \in \Phi_{\fin}$, and that $x+y \in K_{R}(\alpha_i + \alpha_j) =  K_{R}(\alpha)$. In addition, it is $s_\alpha(\alpha +x) = - \alpha  +x$, which implies, as $-\alpha \in \Phi_{\fin}$, that  $-x$ is in $K_{R}(\alpha) $. This shows
	\medskip\\
(2.)  $K_{R}(\alpha) $ is a sublattice of $L(\Phi) \cap R$.
	\medskip\\
	Let $\widehat{\alpha} \in \Phi$. By the other inclusion, there are $\alpha \in \Phi_{\fin}$ and $x \in R$ such that $\widehat{\alpha}  = \alpha +x$. Then $x$ is in $K_R(\alpha) = K_R(\alpha_1) $, and it follows from (1.) and (2.) that
 \medskip\\
	(3.) $L(\Phi) = L(\Phi_{\fin}) + K_R(\alpha_1)$ and $L(\Phi) \cap R = K_R(\alpha_1)$.
		\medskip\\
Lemma~\ref{RadicalLattice} concludes the assertion.
\end{proof}

\begin{remark}
If the finite root system $\FRS$ of Proposition~\ref{RootSystem} is of type $A_{1}$, that is, the corresponding simply-laced elliptic root system is of rank three, then there exist exactly two simply-laced elliptic root systems as described in \cite{Sai85} or in  \cite[Proposition 4.2 and Table 4.5]{BA97}. Therefore, the conclusion of Proposition~\ref{RootSystem} is not valid if the rank of $\Phi$ is less than four.
\end{remark}

\subsection{An  elliptic root basis}\label{Sec:EllipticRootBasis}	
We continue our notation that has been  introduced before. Further, let 
	$$\widetilde{\alpha} := -\alpha_0 + b = \sum_{i = 1}^n m_i \alpha_i$$
be the highest root in $\FRS$ with respect to the fundamental system $B_{\text{fin}} = \{\alpha_1, \ldots , \alpha_n\}$. By construction $B_b  = \{\alpha_0, \ldots , \alpha_n\}$ is a  fundamental system of the affine root system  $\Phi_b$. The Dynkin diagram of $\Phi_b$ is one of the diagrams $X_n^{(1)}$ given in Table Aff 1 of \cite{Kac83} where $X_n$ is one of the simply-laced types $A_n$ ($n \geq 2$), $D_n$ ($n \geq 4$) or $E_n$ ($n \in \{ 6,7,8 \}$).
	
Next, we define a root basis for an elliptic root system (see also \cite[Section 8]{Sai85}).
 
\begin{Definition}\label{DefinitionRootBasis}
Let $\Phi$ be a simply-laced   elliptic root system. For $\alpha \in  B_{\text{aff}}$ set $\alpha^\star = \alpha + a$, and define 
	$$m_{\text{max}}  := \max \{ m_{i} \mid 0 \leq i \leq n \},~\mbox{where}~b = \sum_{i=0}^n m_i \alpha_i, ~\mbox{and}~ B_{\max}:=\{ \alpha_{i^\star} \mid m_{i}= m_{\max} \}.$$
Then the set $B:= B(\Phi):= B_{\text{aff}} \cup B_{\max}$ is called \defn{ elliptic} root basis for $\Phi$.
\end{Definition}
	
Note that an elliptic root basis $B$ is a basis of $V$ if and only if  $|B_{\max}| = 1$ (i.e. if the elliptic root system is of codimension $1$, see \cite[Section~(8.1)]{Sai85}). If the left hand side holds, then we  call $\Phi$ a  \defn{ tubular elliptic} root system and  $B$ a \defn{tubular elliptic} root basis. From now on, we assume that $\Phi$ and $B$ are  tubular. Let $t \in \{1, \ldots , n\}$ be such that $m_t = m_{\text{max}}$. If $B$ is a tubular elliptic root basis of $V$, then we  will call $(V,B, (- \mid -))$ a \defn{tubular} elliptic space.
From \cite[Table~Aff 1 ]{Kac83}  we derive the following types.

\begin{Lemma}\label{Lem:TubEllRoot}
Let $\Phi$ be a tubular elliptic root system. Then $\FRS$ is of type $D_4$ or $E_6, E_7$ or $E_8$.
\end{Lemma}

We conclude the section with another definition.

\begin{Definition}\label{Def:TypeElliptic}
Let $\Phi$ be a tubular elliptic root system. If $\Phi_{\fin}$ is of type $X_n$, then we say that $\Phi$ is of type $X_n^{(1,1)}$.
\end{Definition}

\subsection{A diagram for an elliptic  root system}
Let $(\Phi,B)$ be an elliptic root system with elliptic root basis $B$. Next, we introduce the \defn{elliptic  Dynkin diagram} for $\Phi$. It is defined as the undirected graph on the set of vertices $M:= \{0, \ldots , n, t^{\star}\}$. Let $\alpha_i, \alpha_j \in B$, where $i, j \in M$  with $i \neq j$. If it is $(\alpha_i \mid \alpha_j) = 0$, then there is no edge between $i$ and $j$. There is a single edge between $i$ and $j$ if $(\alpha_i \mid \alpha_j) = -1$ and two dotted edges if $(\alpha_i \mid \alpha_j) = 2$.

Therefore, the elliptic Dynkin diagrams for a tubular system  are precisely the diagrams in Figure~\ref{def:GenCoxDiag}.

\begin{figure}[h]
  \centering
  \begin{tikzpicture}[scale=1.9]

    \node (03) at (0.5,0.5) [] {$D_4^{(1,1)}$};
    \node (G3) at (2,1) [circle, draw, fill=black!50, inner sep=0pt, minimum width=4pt]{};
    \node (A3) at (1.5,0.5) [circle, draw, fill=black!50, inner sep=0pt, minimum width=4pt]{};
    \node (B3) at (2.5,0.5) [circle, draw, fill=black!50, inner sep=0pt, minimum width=4pt]{};
    \node (C3) at (2,0.5) [circle, draw, fill=black!50, inner sep=0pt, minimum width=4pt]{};
    \node (E3) at (1.6, 0.1) [circle, draw, fill=black!50, inner sep=0pt, minimum width=4pt]{};
    \node (F3) at (2.4, 0.1) [circle, draw, fill=black!50, inner sep=0pt, minimum width=4pt]{};

    \node (04) at (0.5,-1) [] {$E_6^{(1,1)}$};
    \node (A4) at (1,-1) [circle, draw, fill=black!50, inner sep=0pt, minimum width=4pt] {};
    \node (B4) at (1.5,-1) [circle, draw, fill=black!50, inner sep=0pt, minimum width=4pt]{};
    \node (C4) at (2,-1) [circle, draw, fill=black!50, inner sep=0pt, minimum width=4pt]{};
    \node (D4) at (2.5,-1) [circle, draw, fill=black!50, inner sep=0pt, minimum width=4pt]{};
    \node (E4) at (3,-1) [circle, draw, fill=black!50, inner sep=0pt, minimum width=4pt]{};
    \node (F4) at (2.4,-1.4) [circle, draw, fill=black!50, inner sep=0pt, minimum width=4pt]{};
    \node (G4) at (2.8,-1.8) [circle, draw, fill=black!50, inner sep=0pt, minimum width=4pt]{};
    \node (H4) at (2,-0.5) [circle, draw, fill=black!50, inner sep=0pt, minimum width=4pt]{};

    \node (05) at (4.2,0.5) [] {$E_7^{(1,1)}$};
    \node (A5) at (4.7,0.5) [circle, draw, fill=black!50, inner sep=0pt, minimum width=4pt] {};
    \node (B5) at (5.2,0.5) [circle, draw, fill=black!50, inner sep=0pt, minimum width=4pt]{};
    \node (C5) at (5.7,0.5) [circle, draw, fill=black!50, inner sep=0pt, minimum width=4pt]{};
    \node (D5) at (6.2,0.5) [circle, draw, fill=black!50, inner sep=0pt, minimum width=4pt]{};
    \node (E5) at (6.7,0.5) [circle, draw, fill=black!50, inner sep=0pt, minimum width=4pt]{};
    \node (H5) at (6.6,0.1) [circle, draw, fill=black!50, inner sep=0pt, minimum width=4pt]{};
    \node (F5) at (7.2,0.5) [circle, draw, fill=black!50, inner sep=0pt, minimum width=4pt]{};
    \node (G5) at (7.7,0.5) [circle, draw, fill=black!50, inner sep=0pt, minimum width=4pt]{};
    \node (I5) at (6.2,1) [circle, draw, fill=black!50, inner sep=0pt, minimum width=4pt]{};

    \node (06) at (4.2,-1) [] {$E_8^{(1,1)}$};
    \node (A6) at (4.7,-1) [circle, draw, fill=black!50, inner sep=0pt, minimum width=4pt] {};
    \node (B6) at (5.2,-1) [circle, draw, fill=black!50, inner sep=0pt, minimum width=4pt]{};
    \node (C6) at (5.7,-1) [circle, draw, fill=black!50, inner sep=0pt, minimum width=4pt]{};
    \node (D6) at (6.2,-1) [circle, draw, fill=black!50, inner sep=0pt, minimum width=4pt]{};
    \node (E6) at (6.7,-1) [circle, draw, fill=black!50, inner sep=0pt, minimum width=4pt]{};
    \node (F6) at (7.2,-1) [circle, draw, fill=black!50, inner sep=0pt, minimum width=4pt]{};
    \node (G6) at (7.7,-1) [circle, draw, fill=black!50, inner sep=0pt, minimum width=4pt]{};
    \node (H6) at (8.2,-1) [circle, draw, fill=black!50, inner sep=0pt, minimum width=4pt]{};
    \node (I6) at (5.7,-0.5) [circle, draw, fill=black!50, inner sep=0pt, minimum width=4pt]{};
    \node (J6) at (6.1,-1.4) [circle, draw, fill=black!50, inner sep=0pt, minimum width=4pt]{};


    \draw[-] (A3) to (C3);
    \draw[-] (C3) to (B3);
    \draw[-] (C3) to (E3);
    \draw[-] (C3) to (F3);
    \draw[-] (G3) to (A3);
    \draw[-] (G3) to (B3);
    \draw[-] (G3) to (E3);
    \draw[-] (G3) to (F3);
    \draw[dashed] ([xshift=0.5]C3.north) to ([xshift=0.5]G3.south);
    \draw[dashed] ([xshift=-0.5]C3.north) to ([xshift=-0.5]G3.south);

    \draw[-] (A4) to (B4);
    \draw[-] (B4) to (C4);
    \draw[-] (C4) to (D4);
    \draw[-] (D4) to (E4);
    \draw[-] (C4) to (F4);
    \draw[-] (F4) to (G4);
    \draw[dashed] ([xshift=0.5]C4.north) to ([xshift=0.5]H4.south);
    \draw[dashed] ([xshift=-0.5]C4.north) to ([xshift=-0.5]H4.south);
    \draw[-] (B4) to (H4);
    \draw[-] (D4) to (H4);
    \draw[-] (F4) to (H4);

    \draw[-] (A5) to (B5);
    \draw[-] (B5) to (C5);
    \draw[-] (C5) to (D5);
    \draw[-] (D5) to (E5);
    \draw[-] (E5) to (F5);
    \draw[-] (F5) to (G5);
    \draw[-] (D5) to (H5);
    \draw[-] (I5) to (C5);
    \draw[-] (I5) to (E5);
    \draw[-] (I5) to (H5);
    \draw[dashed] ([xshift=0.5]D5.north) to ([xshift=0.5]I5.south);
    \draw[dashed] ([xshift=-0.5]D5.north) to ([xshift=-0.5]I5.south);

    \draw[-] (A6) to (B6);
    \draw[-] (B6) to (C6);
    \draw[-] (C6) to (D6);
    \draw[-] (D6) to (E6);
    \draw[-] (E6) to (F6);
    \draw[-] (F6) to (G6);
    \draw[-] (G6) to (H6);
    \draw[-] (I6) to (B6);
    \draw[-] (I6) to (D6);
    \draw[-] (I6) to (J6);
    \draw[-] (C6) to (J6);
    \draw[dashed] ([xshift=0.5]C6.north) to ([xshift=0.5]I6.south);
    \draw[dashed] ([xshift=-0.5]C6.north) to ([xshift=-0.5]I6.south);

  \end{tikzpicture}
  \caption{Elliptic Dynkin diagrams for the tubular elliptic root systems} \label{def:GenCoxDiag}
\end{figure}

\begin{Proposition}[{\cite[Theorem 9.6]{Sai85}}] \label{thm:SaitoClassification}
	Let $(\Phi,B)$ be an elliptic root system with elliptic root basis $B$. The elliptic Dynkin diagram for $(\Phi, B)$ is uniquely determined by the isomorphism class of $\Phi$. Conversely, the elliptic Dynkin diagram for $\Phi$ uniquely determines the isomorphism class of $(\Phi, B)$.
\end{Proposition}

\subsection{The tubular elliptic Weyl system $(W,S)$}\label{Subsec:EllipticWeylgroup}
	
\begin{Definition} \label{def:basic}
	Let $(V, B, (- \mid -))$ be a tubular  elliptic space.
	\begin{enumerate}
	\item[(a)] For $\alpha_i \in B$ we abbreviate $s_{\alpha_i}$ by $s_i$.
	\item[(b)] For $\Psi \subseteq \Phi$ let $\displaystyle W_{\Psi}: = \langle s_{\alpha} \mid \alpha \in \Psi \rangle$ be the  group  that is generated by the reflections of $V$ with respect to the hyperplanes $\alpha^\bot$, where $ \alpha$ is in  $\Psi $, and let $W:=W_\Phi$.
	\item[(c)] The set $S:=\lbrace s_{\alpha}\mid \alpha \in B \rbrace$ is called \defn{simple system} and its elements  \defn{simple reflections}. We call $(W,S)$ a \defn{tubular elliptic Weyl system}  and $W$  \defn{tubular elliptic  Weyl group}.
	\item[(d)] 			 
	The set $T:=\bigcup_{w\in W}wSw^{-1}$ is the \defn{set of reflections} for $(W,S)$.
	\item[(e)]  An element $c:=\left(\prod_{i \in M\setminus \lbrace t, t^{\star} \rbrace}s_{i}\right)\cdot s_{t}s_{t^{\star}}$ is called \defn{Coxeter transformation} where we take the first $|B|-2$ factors in arbitrary order.
	\end{enumerate}
\end{Definition}

\begin{lemdef}\label{LemDef:Transvection} Let $\alpha \in \Phi_{\fin}$, $x \in \{a,b\}$ and $k \in \ZZ$. Set $\TR_x(k\alpha) := s_\alpha s_{\alpha + kx}$. Then $\TR_x( k \alpha)$ is the transvection $v \mapsto v - (v \mid  \alpha) kx$.
\end{lemdef}
\begin{proof}
    This is a direct calculation.
\end{proof}

The next lemma is a consequence of the definition of $B$ (see \cite{Kac83}). 
	
\begin{Lemma}
Let $(\Phi,B)$ be an elliptic root system with elliptic root basis $B$. Then the following holds.
	\begin{enumerate}
		\item[(a)] $W= W_{B}$;
		\item[(b)] $\Phi = W_{B}(B)$.
	\end{enumerate}
\end{Lemma}

\begin{proof}
We first show (b). By Lemma~\ref{Lem:Types} the generalized root system $\Phi$ contains the affine root system $\Phi_b$. Let $\alpha \in B_{\fin}$ be such that 
$(\alpha \mid \alpha_t) = -1$. Then for $k, \ell \in \ZZ$ we have  $\alpha +\ell b \in \Phi_b$ and $\TR_a(\alpha_t)^k(\alpha +\ell b) = \alpha + k a +\ell b \in W_B(B)$.  We conclude assertion (b) from Lemma~\ref{RootSystem} by observing that  $W_{\fin} = \langle s_{\beta} \mid \beta \in B_{\fin} \rangle \leq W_B$ acts transitively on $\Phi_{\fin}$ and that the elements in 
$W$ fix the elements in $R$.

	If $\gamma$ is in $\Phi$, then there are $w \in W_B$ and $\alpha \in B$  such that  $w(\alpha) = \gamma$ by (b). Then $s_{\gamma} = s_{w(\alpha)} = ws_\alpha w^{-1}$ is in $W_B$, and therefore $W = W_B$, which is (a).
\end{proof}

In the following observation, we use the fact that $\Phi$ is a crystallographic  root system.

\begin{Lemma}\label{lem:crystallographic}
If $\emptyset \neq \Psi \subseteq \Phi$, then $w(L(\Psi)) \subseteq L(\Psi)$ holds for every $w \in W_\Psi$.
\end{Lemma}
\begin{proof}
If we apply $s_\alpha$ to a vector $v \in V$, then we only change $v$
by subtracting $\frac{2(v \mid \alpha)}{(\alpha\mid \alpha)} \alpha$ from $v$, and $\frac{2(v \mid \alpha)}{(\alpha\mid \alpha)} \alpha$ is an integer by Definition~\ref{def:root_system}.
\end{proof}

\subsection{The affine Coxeter subgroups $W_a$ and $W_b$ of $W$}\label{CoxSubgroupWb}
Let $\Gamma_{\fin}$ and $\Gamma_b = \Gamma_{\aff}$ be the subdiagrams of $\Gamma$ with set of vertices $M\setminus{\{0,t^\star\}}$ and   $M\setminus{\{t^\star\}}$, respectively, and set
$$W_{\fin}:= \langle s_i \mid 1 \leq i \leq n\rangle ~\mbox{and}~W_b:= \langle s_i \mid 0 \leq i \leq n \rangle .$$ 
Then $W_{\fin} \leq W_b \leq W$. Notice that the diagram $\Gamma_{b}$ is always a star (see Subsection~\ref{Sec:Notation}). We say that a Coxeter system whose diagram is a star is of \defn{star type}.

\begin{Proposition}\label{Prop:StructureSubsystems}
The following holds.
\begin{itemize}
	\item[(a)] It is $(W_{\fin}, S_{\fin})$ a finite Coxeter system of type $\Gamma_{\fin}$ with simple system $S_{\fin} := \{s_i  \mid 1 \leq i \leq n\}$ and root system $\FRS$.
	\item[(b)] It is $(W_b, S_b)$ an affine Coxeter system of type $\Gamma_{\aff}$ with simple system $S_b:= S_{\aff}:=  \{i \mid 0 \leq i \leq n\}$, root system  $\Phi_b = \Phi_{\aff} = \Phi_{\fin} + \ZZ b$  and linear representation $W_b \rightarrow $ $\GL(V_{b}), s_i \mapsto s_i\mid _{V_b}$.
    \item[(c)] Every affine Coxeter system of star type appears as a Coxeter system $(W_b, S_b)$  in (b).
    \end{itemize}
\end{Proposition}
\begin{proof} (a) and (b) are consequences of Lemma~\ref{Lem:Types}, and (c) of the construction.
\end{proof}

\begin{remark}\label{WaCoxeter}
Set $S_a:= \{s_i , s_{t^\star}  \mid 1 \leq i \leq n \}$, 
$W_a:= \langle S_a \rangle  \leq W$ and let $V_a$ be the subspace of $V$ generated by $\{ \alpha_i, \alpha_{t^\star}\mid 1 \leq i \leq n \}$.  
Then $(- \mid -)$ restricted to $V_a$ is positive  semidefinite  with one dimensional radical, and $\Phi_a:= \Phi \cap V_a = \Phi_{\fin} + \ZZ a$ is a generalized root system in $V_a$. Therefore, $(W_a, S_a)$ is an affine Coxeter system (see  \cite[Proposition~4.7 (e)]{Kac83}). From the fact that $\Phi_a/U$ is isomorphic to $\Phi_{\fin}$, we conclude that $W_a$ is also an affine Coxeter system of type $\Gamma_b = \Gamma_{\aff}$.
\end{remark}

\begin{lemdef}\label{Lem:StructurAffineCox}
Let $x \in \{a,b\}$ and $\lambda = \sum_{i=1}^{n} c_i \alpha_i \in  L(\FRS)$.
\begin{itemize}
\item[(a)]  Let $\TR_x(\lambda) := \prod_{i = 1}^k \TR(c_i\alpha_i)$.
Then $\TR_x(\lambda)$ is the map $\displaystyle \TR_x(\lambda) : V \rightarrow V, v \mapsto v - (v \mid  \lambda)x$.
\item[(b)] The subgroup $N(\Gamma_{\fin},x) := \langle \TR_x(\alpha_i) \mid 1 \leq i \leq n \rangle$ of $W$ is isomorphic to the abelian group $L(\Phi_{\fin})$, and therefore 
free abelian of rank $n$.	
\item[(c)] $N(\Gamma_{\fin},x)$ is a normal subgroup of $W_x$, and $W_{\fin}$ acts on 	$N(\Gamma_{\fin},x) $ by $w \TR_x(\lambda) w^{-1} = \TR_x(w(\lambda))$ for $w \in W_{\fin}$.
\item[(d)] $N(\Gamma_{\fin}, a)$ and $N(\Gamma_{\fin}, b)$ commute.
\item[(e)] $W_x$ is a semi-direct product of $N(\Gamma_{\fin},x)$ by $W_{\fin}$, that is the short exact sequence 
$$1 \rightarrow  N(\Gamma_{\fin},x) \rightarrow W_x  \rightarrow W_{\fin} \rightarrow 1$$
splits.
\item[(f)] $W$ is a semi-direct product of $N(\Gamma_{\fin},a)$ by $W_{b}$, that is the short exact sequence 
$$1 \rightarrow  N(\Gamma_{\fin},a) \rightarrow W \rightarrow W_{b} \rightarrow 1$$
splits.
\end{itemize}
\end{lemdef}
\begin{proof}
 It is   $\TR_x(\lambda) \TR_x(\mu) = \TR_x(\lambda +\mu)$
 for $\lambda, \mu \in L(\Phi_{\fin})$, which implies (a) and that
 $N(\Gamma_{\fin},x)$ is isomorphic to $L(\alpha_1, \ldots , \alpha_n) = L(\FRS)$, which is free abelian of rank $n$. 
 Assertion (c) follows from the fact that $w s_{\alpha_i} w^{-1}  = s_{w(\alpha_i)}$ for all $w \in W_{\fin}$. Assertion (d) follows from the fact that $N(\Gamma_{\fin}, a)$ and $N(\Gamma_{\fin}, b)$ 
 are two normal subgroups, which commute. The last assertion is a direct consequence of Proposition~\ref{Prop:StructureSubsystems} and Remark~\ref{WaCoxeter}(see also \cite{Hum90}).
 \end{proof}

We conclude from Lemma~\ref{Lem:StructurAffineCox} that every element $w \in W_b$ can be written in normal form $w = w_{\fin} t_b(\lambda)$ for some $w_{\fin} \in W_{\fin}$ and some $\lambda \in L(\Phi_{\fin})$; written in vector notation as 
$$w = \left( \begin{array}{c} w_{\fin} \\ \lambda 
\end{array} \right).$$

The multiplication  for $w_1, w_2 \in W_{\fin}$ and $\lambda, \nu \in L(\Phi_{\fin})$ is given by 
$$\left( \begin{array}{c} w_{1} \\ \lambda 
\end{array} \right) \left( \begin{array}{c} w_{2} \\ \nu
\end{array} \right) = \left( \begin{array}{c} w_{1} w_2 \\ w_2^{-1}(\lambda) + \nu \end{array} \right).$$

\section{The  hyperbolic cover of an elliptic Weyl system}\label{Sec:HyperbolicCover}
\subsection{The hyperbolic extended space}\label{Sec:HyperbolicExtended}

Let $(\hat{V}, (- \mid -))$ be an $\RR$-space with signature $(n+2, 2, 0)$. Then, due to Sylvester's law of inertia, there is a basis  $X$  of $\hat{V}$ such that the Gram matrix of $(- \mid -)$ has diagonal form and its first $n+2$ diagonal entries are positive and its last $2$ entries are negative. Thus, $\hat{V} = V_+ \oplus V_-$ is the orthogonal direct sum of $V_+$,  the $\RR$-span of the first $n+2$ basis vectors and $V_-$,  the $\RR$-span of the last two. Then $(v\mid v) > 0$ for every non-zero vector in $V_+$ and $(v\mid v) < 0$ for every non-zero vector in $V_-$.
Recall that two vectors $v_1,v_2 \in V$ form a  \defn{hyperbolic pair}, if $(v_i \mid v_i) = 0$, for $i = 1,2$  and  $(v_1\mid v_2) = 1$.

\begin{Lemma}\label{HyperbolicExtension}
$(\hat{V},  (- \mid -))$ contains a subspace $V$ that is elliptic and of dimension $n+2$.
\end{Lemma}
\begin{proof} 
Let $u_1 \in V_-$ and $u_2 \in V_+$. Then there are $\lambda_i \in \RR$ such that $(u_i \mid u_i )  =(-1)^{i}\lambda_i$ for $i = 1,2$. Let $\mu_1, \mu_2 \in \RR$ such that $\mu_1^2 =  \lambda_2$ and $\mu_2^2 =  \lambda_1$ and set
$$\tilde{v}_1:= \mu_1 u_1 + \mu_2 u_2~\mbox{and} ~\tilde{v}_2:= \mu_1 u_1  - \mu_2 u_2.$$
Then $\tilde{v}_1$ and $\tilde{v}_2$  are isotropic. Let $v_1 := \tilde{v}_1$ and $v_2:=-1/(2\lambda_1\lambda_2)\tilde{v}_2$. 
Then $v_1$ and $v_2$ are isotropic vectors such that $(v_1 \mid v_2) = 1$, that is, $(v_1,v_2)$ is a hyperbolic pair in the space $V_1:= \langle v_1, v_2\rangle = \langle u_1, u_2 \rangle$, and  $V_1$ and  $V_1^\bot$ are spaces of signature $(1,1,0)$ and $(n+1,1,0)$, respectively.  We can repeat the procedure in $ V_1^\bot$ and obtain an orthogonal  decomposition  of $\hat{V} = V_1 \bot V_2 \bot V_3$, where the signature of  $V_2$ is $(1,1,0)$ and $V_3:= (V_1  +V_2)^\bot$ has signature $(n,0,0)$.  Let $(v_3,v_4)$ be a hyperbolic pair in $V_2$, and $a:= v_1$ and $b:= v_3$. Then the span of  $a$, $b$ and of $V_3$ is a subspace $V$, as stated.
\end{proof}

\begin{remark}\label{Rem:Extension}
Note that $\hat{V}$ is up to isometry the non-degenerate vector space of smallest dimension that contains $(V, (- \mid -))$ as a subspace: let $(V^\prime, ( -\mid -))$ be a non-degenerate $\RR$-space that contains $(V, (- \mid -))$ as a subspace. Then $a,b \in (V_{\fin} +\RR b)^\perp$
Since $V^\prime$ is non-degenerate, it is $a^\bot \neq b^\bot$, and  there is a 
 vector $v$ in $(V_{\fin}+\RR b)^\perp$ that is not perpendicular to $a$. 
 As $a$ is an isotropic vector,
 the signature of $V_1:= \langle a,v \rangle$ is $(1,1,0)$. Repeating this argument for $b \in (V_3 + V_1)^\perp$  where $V_3:= V_{\fin}$ we obtain a subspace $V_2:= \langle b,u \rangle \leq (V_3 + V_1)^\perp$ of signature $(1,1,0)$, and therefore $\hat{V}$ and $V_1 \bot V_2 \bot V_3\leq V^\prime$  are isometric.
\end{remark}

We will continue to use the notation introduced in 
Remark~\ref{Rem:Extension}.

\begin{Lemma}
Let $\tilde{V}:= \langle a \rangle \bot V _2 \bot V_3 $ and let $b^\prime \in V_2$ be such that $(b,b^\prime)$ is a hyperbolic pair. Then $$\tilde{V} = V \oplus \langle b'\rangle, $$ and the signature of $\tilde{V}$ equipped with the bilinear form on $\hat{V}$ restricted to $\tilde{V}$ is  $(n+1,1,1)$.
\end{Lemma}
\begin{proof} 
The assertion follows from the fact that $V_2 = \langle b,b^\prime \rangle$.
\end{proof}

We call $(\tilde{V}, (-\mid-))$ and $(\hat{V}, (-\mid -))$ \defn{hyperbolic extended spaces} of $V$.

\subsection{The hyperbolic covers $(\tilde{W}, \tilde{S})$ and $(\hat{W}, \hat{S})$ of $(W$,S)}\label{sec:HypCovers}
Let $\tilde{s}_\alpha$  be the reflection in $\tilde{V}$ with respect to the root $\alpha \in V \subset \tilde{V}$, and abbreviate $\tilde{s}_{\alpha_i}$  by $\tilde{s}_i$. Considering $\alpha \in V \subset \tilde{V}  \subset \hat{V}$ define $\hat{s}_{\alpha_i}$ and $\hat{s}_i$ analogously to $\tilde{s}_{\alpha_i}$ and $\tilde{s}_i$.

   \begin{Definition}\label{Def:HypCover}
   We set
   \begin{itemize}
\item[(a)]  $ \tilde{S}  := \{\tilde{s}_i , \tilde{s}_{t^\star} \mid 0 \leq i \leq n \}~\mbox{and}~\tilde{W} :=  \langle \tilde{S} \rangle$, and 
\item[(b)] $ \hat{S}  := \{\hat{s}_i , \hat{s}_{t^\star} \mid 0 \leq i \leq n \}~\mbox{and}~\hat{W} :=  \langle \hat{S} \rangle.$
\end{itemize}     
We call $(\tilde{W}, \tilde{S})$ and $(\hat{W}, \hat{S})$ the \defn{hyperbolic covers} of $(W,S)$.
   \end{Definition}

\begin{remarkOHNE}
 We call $\tilde{W}$ and $\hat{W}$ covers of $W$ as they are central extensions of $W$ (see Proposition~\ref{StructureTildeW}). 
\end{remarkOHNE}

 Next, we recall a tool that is helpful in  understanding $\tilde{W}$ and $\hat{W}$ (see \cite{Sai85}).

\subsection{The Eichler-Siegel map}
Recall that $U = \RR a \leq R$, and  consider the \defn{Eichler-Siegel map} $\tilde{E}$ on $\tilde{V}$ (see \cite[Section (1.14)]{Sai85})
$$\tilde{E}: V \otimes_\ZZ  V/U \rightarrow \END(\tilde{V} ), ~ \sum_i f_i \otimes g_i \mapsto \left(v \mapsto v- \sum_i  (g_i \mid v)f_i \right).$$
As in \cite[Section (1.14)]{Sai85} we define a semi-group structure $\circ$ in $V \otimes_{\ZZ}  V/U$ by
$$\varphi_1 \circ \varphi_2:= \varphi_1 + \varphi_2 - (\varphi_1 \mid \varphi_2),$$
where $(\varphi_1 \mid \varphi_2) := \sum_{i_1,i_2} f_{i_1}^1 \otimes (g_{i_1}^1 \mid f_{i_2}^2) g_{i_2}^2$ for $\varphi_j = \sum_{i_j} f_{i_j}^j \otimes g_{i_j}^j$, for $j = 1,2$, is a non-symmetric bilinear form on  $V \otimes_{\ZZ}  V/U$.
We recall some properties of the map $\tilde{E}$. All the proofs are given in
\cite{Sai85}.

\begin{Lemma}[{\cite[Sections (1.14) and (1.15)]{Sai85}}] \label{Eichler} 
The following hold.
\begin{itemize}
\item[(a)] $\tilde{E}$ is injective.
\item[(b)] $\tilde{E}$ is a homomorphism of semi-groups that is $\tilde{E}(\varphi \circ \psi) = \tilde{E}(\varphi) \tilde{E}(\psi)$.
\item[(c)] the subspace $R \otimes  V/U $ of $V \otimes V/U$ is closed under $\circ$, and  $\circ$ and $+$ agree on $R \otimes V/U $.
\item[(d)] It is  $\tilde{E}(u_1 \otimes u_2 + v_1 \otimes v_2 ) = \tilde{E}(u_1 \otimes u_2) \tilde{E}(v_1 \otimes v_2)$, whenever $(u_2 \mid v_1)  = 0$.
\item[(e)] $\tilde{s}_{\alpha} = \tilde{E}(\alpha \otimes \alpha)$ for all non-isotropic $\alpha \in V$.
\item[(f)] The inverse of $\tilde{E}$ on $\tilde{W}$ is well-defined: $\displaystyle \tilde{E}^{-1}: \tilde{W} \rightarrow V \otimes V/U.$
\item[(g)]  $\tilde{E}^{-1}(\tilde{W}) \subseteq L(\Phi) \otimes_{\ZZ} (L(\Phi) / U)$;
\end{itemize}
\end{Lemma}

The next assertions can be checked by direct calculation with the Eichler map.

\begin{Lemma}\label{normalb}
The following holds.
\begin{itemize}
\item[(a)] $\tilde{E}(L(\Phi_b) \otimes b)$  is a normal subgroup of $\tilde{E}( L(\Phi_{b}) \otimes L(\Phi/U))$.
\item[(b)] $w \in \tilde{W}_{\fin}:= \langle \tilde{s}_i \mid 1 \leq i \leq n \rangle$ acts on $\tilde{E}(L(\Phi_{\fin}) \otimes b)$ by sending $\tilde{E}(\beta \otimes b)$ to $\tilde{E}(w^{-1}(\beta) \otimes b)$. 
\end{itemize}
\end{Lemma} 
\begin{proof}
    The proof of (a) and (b) is a direct calculation using Lemma~\ref{Eichler} (b) and (d).
\end{proof}

\begin{remark}
We are interested in the group $\tilde{W}$, which is generated by a subset of $\tilde{E}(L(\Phi) \otimes_{\ZZ} (L(\Phi) / U))$ by Lemma~\ref{Eichler} (g). Since we identify  $V/U$ with $V_b$ and $\Phi/U$ with $\Phi_b$, we will also consider $\tilde{E}$ restricted to $L(\Phi) \otimes  L(\Phi_{b})$. We also identify $V/R$ with $V_{\fin} \subset V_b$. This enables us to embed $V/R$ into $V/U$.
\end{remark}

\begin{Lemma} \label{EichlerV}
The following holds.
\begin{itemize}
\item[(a)] $(\tilde{s}_\alpha)_{\mid V} = s_\alpha$ for every $\alpha \in \Phi$, and $\varphi: \tilde{W} \rightarrow W$ with $w \mapsto w_{\mid V}$ is an epimorphism.
\item[(b)] The Eichler-Siegel map on $V$ is given by $\displaystyle E:  V \otimes V/R  \rightarrow \END (V),~x \mapsto \tilde{E}(x)_{\mid V}.$
\item[(c)]  All statements of  \ref{Eichler} hold in adapted form  for $E$, as well.
\end{itemize} 
\end{Lemma}
\begin{proof}
 Assertion (a) is a direct consequence of the definition of   $\tilde{s}_\alpha$, and  of the fact that $\tilde{s}_\alpha(V) \subseteq V$. The map introduced in (b) is the analog of the Eichler map for $\tilde{V}$ for $V$ (see also \cite[Section (1.14)]{Sai85}). 

Finally, we prove (c). Let $\varphi \in V \otimes V/R$. Then it is $\tilde{E}(\varphi) (b') = b'$ for all $v \in V$. This shows that the injectivity of $\tilde{E}$ also implies the injectivity of $E$. The analogs of (b)-(g) of Lemma~\ref{Eichler} for $E$ hold as $E(\varphi)$ is the restriction of $\tilde{E}(\varphi )$ to $V$.
\end{proof}

\medskip
\noindent Finally, we introduce the Coxeter transformation 
$$\tilde{c} = \tilde{s}_1 \cdots  \tilde{s}_{t-1}  \tilde{s}_{t+1} \cdots  \tilde{s}_n   \tilde{s}_0 \tilde{s}_t  \tilde{s}_{t^\star} \in \tilde{W},$$ 
which will be properly studied in Section~\ref{CoxeterTransMulti}.

\begin{Lemma} \label{center} 
The following holds.
\begin{enumerate}
\item[(a)] $[\tilde{W},  \tilde{E}((\ZZ a + \ZZ b)\otimes b) ] = 1$; 
\item[(b)]  $\tilde{c}^m = \tilde{E}(a \otimes b)$, where $m$ is the order of $c$ in $W$;
\item[(c)] $[\tilde{V}, \tilde{W}] \leq V$ and $\tilde{W}_{\mid R} = 
\langle \idop_R \rangle$.
\end{enumerate}
\end{Lemma}
\begin{proof}
Statements (a) and (b) follow from direct calculation with the Eichler map, while (c) is a consequence of the fact that $\tilde{W}$ is generated by the reflections $\tilde{s}_\alpha$, where $\alpha \in \Phi \subset V$, and that $\tilde{s}_\alpha(b^\prime)$ is in $V$
for every $\alpha \in \Phi$.
\end{proof}

\begin{Proposition}\label{StructureTildeW}
The following holds.
\begin{itemize}
	\item[(a)] The group $\tilde{W}$ is a non-split central extension of $W$ by $\langle 	\tilde{E}(a \otimes b) \rangle  = \tilde{E}(a \otimes \ZZ b)  \cong \ZZ$.
	\item[(b)] The groups $\tilde{W}$ and $\hat{W}$  are isomorphic.
	\end{itemize} 
\end{Proposition}
\begin{proof}
Let $J$ be the Gram matrix of the bilinear form $(-\mid -)$ on $\tilde{V}$ with respect to a basis $B$, and let $A$ be the matrix with respect to $B$ of an element in $\tilde{W}$, which acts trivially on $V$. As $(Av_1 \mid Av_2) = (v_1,v_2)$ for all $v_1,v_2 \in \tilde{V}$, it is $A^t JA = J$. Taking into account Lemma~\ref{center} (c) we derive that $A$ is the matrix of an element in $\langle \tilde{E}(a \otimes b) \rangle$. We know by Lemma~\ref{center} (b) that $\langle \tilde{E}(a \otimes b) \rangle \leq \tilde{W}$. Therefore, we conclude with Lemma~\ref{center} (a) and Lemma~\ref{EichlerV} (a) that $\tilde{W}$ is the central extension given in (a).

As $\tilde{W}$ is generated by elements of order $2$, a splitting $\tilde{W} = \langle \tilde{E}(a \otimes b) \rangle  \rtimes Y (= \langle \tilde{E}(a \otimes b) \rangle  \times Y $) 
for some $Y < \tilde{W}$, would imply the contradiction that
$Y = \tilde{W}$. Hence $\tilde{W}$  is a non-split extension.
	
Considering the Gram matrix for the form $(-\mid -)$ on $\hat{V}$ we calculate that $\hat{W}$ acts faithfully on $\tilde{V}$. This shows that the homomorphism from $\hat{W}$ to $\tilde{W}$ that sends $w \in  \hat{W}$ to $w\mid_{\tilde{V}} \in \tilde{W}$ is an isomorphism, as stated in (b).
\end{proof}

Let $W$ and $W'$ be two groups that are generated by $S$ and $S'$, respectively.
We say that $(W,S)$ and $(W',S')$ are \defn{isomorphic}, if there is an isomorphism from
$W$ to $W'$ that maps $S$ bijectively onto $S'$.

\begin{corollary}
 The hyperbolic covers $(\tilde{W},\tilde{S})$ and $ (\hat{W}, \hat{S})$ are isomorphic. 
\end{corollary}

\begin{remark}\label{Rem:HomotildeWtoW}
We want to highlight  that the map $\varphi: w \mapsto w\mid_{V}$ is an epimorphism from $\tilde{W}$ to $W$, whose kernel is $\tilde{E}(a \otimes \ZZ b)$ by Proposition~\ref{StructureTildeW}.
\end{remark}

\subsection{The Coxeter subgroups $\tilde{W}_b$ and $\tilde{W}_a$ of $\tilde{W}$}
In this section we discuss a preimage of the affine Coxeter system $(W_b, S_b)$ in $\tilde{W}$.
We will  abbreviate $\tilde{s}_{\alpha_i}$ by  $\tilde{s}_i$.


\begin{Lemma}\label{lem:IsoAffUG}
Set $\tilde{S}_b  = \{\tilde{s}_i \mid 0 \leq i \leq n \}$ 
 and 
$\tilde{W}_b :=  \langle \tilde{S}_b \rangle $. Then 
$(\tilde{W}_b, \tilde{S}_b)$ is a Coxeter system, which is isomorphic to the affine Coxeter system $(W_b, S_b)$. 
An isomorphism is given by $\varphi\mid_{\tilde{W}_b}$, where $\varphi: \tilde{W} \rightarrow W$ is the restriction map $w \mapsto w_{\mid V}$.

\end{Lemma}
\begin{proof} 
Lemma~\ref{Eichler} (a) yields
$$\tilde{W}_b \cap \langle \tilde{E}(a \otimes b) \rangle  \leq \tilde{E}(V_b \otimes V/U) \cap \langle \tilde{E}(a \otimes b) \rangle   = \tilde{E}((V_b \otimes V/U) \cap (a \otimes \ZZ b))= \tilde{E}( 0 \otimes b) = 1.$$
Hence,
$\tilde{W}_b$ acts faithfully on $V_b = V_{\aff} \subseteq \tilde{V}$ by Proposition~\ref{StructureTildeW}.
Therefore, the map $\varphi\mid_{\tilde{W}_b}$ defines an isomorphism  from $\tilde{W}_b$  to $W_b$. 
\end{proof}

\begin{remark}\label{tildeWaCoxeter}
	Let  $\tilde{W}_a:= \langle \tilde{s}_i , \tilde{s}_{t^\star}  \mid 1 \leq i \leq n \rangle  \leq \tilde{W}$. Then $\tilde{W}_a$  acts faithfully on $V_a = \langle \alpha_i, \alpha_t^\star \mid 1 \leq i \leq n \rangle \subseteq \tilde{V}$ and is therefore also an affine Coxeter group of type $\Gamma_b$ (see Remark~\ref{WaCoxeter}), thus $\tilde{W}_a\cong  W_a \cong W_b$.
\end{remark}

\begin{Definition}
 Let $x \in \{a,b\}$.
 \begin{itemize}
     \item[(a)] For $\alpha \in \Phi_{\fin} $, $k \in \ZZ$, set 
$\tilde{\TR}_x(k\alpha): = \tilde{s}_\alpha \tilde{s}_{\alpha +kx}$ and 
$ \tilde{N}(\Gamma_{\fin},x) := \langle \tilde{\TR}_x(\alpha_i)  \mid 1 \leq i \leq n \rangle  \subseteq \tilde{W}$.
      \item[(b)] For $k_i \in \ZZ$ and  $\lambda = \sum_{i = 1}^n k_i\alpha_i  \in L(\Phi_{\fin})$,
set  $\tilde{\TR}_x(\lambda) := \prod_{i = 1}^n  \tilde{\TR}_x(k_i\alpha_i)$.
 \end{itemize}
\end{Definition}	

\begin{remark}
    Notice that $\tilde{\TR}_b(k\alpha)$ is no longer a transvection.
\end{remark}

\begin{Lemma}\label{lem:Translationgroups}
Let $x \in \{a,b\}$.  
\begin{itemize}
\item[(a)] It is $\tilde{N}(\Gamma_{\fin},x) 
= \{\tilde{\TR}_x(\lambda) \mid \lambda \in L(\Phi_{\fin})  \}$
free abelian of rank $n$;
\item[(b)] $\tilde{W}_b = \tilde{W}_{\fin} \ltimes   \tilde{N}(\Gamma_{\fin},b)$, where
$\tilde{W}_{\fin}:= \langle \tilde{s}_i \mid 1 \leq i  \leq n\rangle$;
\item[(c)] $\varphi(\tilde{N}(\Gamma_{\fin},x)) = N(\Gamma_{\fin},x) $ and 
 $\varphi (\tilde{W}_{\fin}) = W_{\fin}$.
\end{itemize}
\end{Lemma}
\begin{proof}
Statement (a) follows from $\tilde{\TR}_x(\lambda) \tilde{\TR}_x(\mu) = \tilde{\TR}_x(\mu) \tilde{\TR}_x(\lambda)=
\tilde{\TR}_x(\lambda+ \mu)$, and (b) and (c) are a direct consequence of Lemma~\ref{lem:IsoAffUG} and  Remark~\ref{tildeWaCoxeter}.
\end{proof}

\subsection{More on the  structure of $\tilde{W}$}\label{sec:StructurehyperbolicCover}
We further explore the structure of $\tilde{W}$ and show that $\tilde{W}$ is not a Coxeter group.
In the next subsection we will conclude that $(\tilde{W}, \tilde{S})$ is isomorphic to an extended Coxeter system (see Section~\ref{Sec:ExtCox}).

For $k \in \ZZ$ set 
$$\tilde{\TR}_a(kb) := \tilde{E}( a \otimes k b) ~\mbox{and}~ \tilde{\TR}_a(k \alpha_0) := \tilde{\TR}_a(-k \tilde{\alpha}) \tilde{\TR}_a(kb),$$ and for $\mu := \sum_{i = 0}^n k_i\alpha_i   \in L(\Phi_b)$ where $ k_i \in \ZZ$ define
$$ \tilde{\TR}_a(\mu) := \prod_{i = 0}^n  \tilde{\TR}_a(k_i\alpha_i).$$

\begin{Lemma}\label{Lem:AffineGItter}
The following holds.
\begin{itemize}
\item[(a)] 
$\tilde{N}(\Gamma_{b},a) := \langle  \tilde{\TR}_a(\alpha_i) \mid 0 \leq i \leq n \rangle =	 \{\tilde{\TR}_x(\lambda) \mid \lambda \in L(\Phi_{b})  \}$ is a normal subgroup of $\tilde{W}$.
\item[(b)] $\tilde{N}(\Gamma_{b},a) = \tilde{N}(\Gamma_{\fin},a) \times  \langle \tilde{\TR}_a(b) \rangle$ is free abelian  of rank $n+1$. 
\item[(c)] $\tilde{N}(\Gamma_{b},a) \cong L(\Phi_b)$ as abelian groups.
\end{itemize}
\end{Lemma}

Now we can depict the structure of $\tilde{W}$.

\begin{Proposition}\label{ExactStructure}
The following holds.
\begin{itemize}
\item[(a)]  $\displaystyle \tilde{W} = \tilde{W}_b \ltimes   \tilde{N}(\Gamma_{b},a)  \cong W_b  \ltimes   L(\Phi_b)$
 and the following short exact sequence splits. 
$$1 \rightarrow  \tilde{N}(\Gamma_{b},a) \rightarrow \tilde{W} \rightarrow W_b \rightarrow 1$$
\item[(b)] 	Let $w \in \tilde{W}_{\fin}$  and $\tilde{\TR}_b(\lambda) \in \tilde{N}(\Gamma_{\fin},b)$. Then   the action of $w$ and  $\tilde{\TR}_b(\lambda)$ on $\tilde{N}(\Gamma_{b},a) $ is given, respectively, by
$$w^{-1} \tilde{\TR}_a(\mu) w = \tilde{\TR}_a(w^{-1}(\mu) )~\mbox{and}~ 
\tilde{\TR}_b(\lambda)^{-1} \tilde{\TR}_a(\mu) \tilde{\TR}_b(\lambda)
=  \tilde{\TR}_a (\tilde{\TR}_b(-\lambda) (\mu)).$$
\end{itemize}
In particular, in (a) the action of $W_b$ on $L(\Phi_b)$ is the natural action
of the Coxeter group $W_b$ on its root lattice $L(\Phi_b)$.
\end{Proposition}
\begin{proof}
The kernel of the map $\varphi: \tilde{W} \rightarrow W$  is a subgroup of $\tilde{N}(\Gamma_{b},a)$ by Remark~\ref{Rem:HomotildeWtoW} and Lemma~\ref{Lem:AffineGItter} (b). This fact and the fact that $\varphi(\tilde{W}_b \ltimes   \tilde{N}(\Gamma_{b},a)) =
W_b \ltimes  N(\Gamma_{b},a)) = W$ by Lemma~\ref{lem:Translationgroups} (c) yield $\displaystyle \tilde{W} = \tilde{W}_b \ltimes   \tilde{N}(\Gamma_{b},a)$.
The isomorphism stated in (a) is a consequence of Lemmas~\ref{lem:IsoAffUG} and \ref{Lem:AffineGItter} (c).
Statement (b) follows from Lemma~\ref{Lem:StructurAffineCox} (c) and a calculation with the Eichler-Siegel map.
\end{proof}		

To complete the picture, we show that $\tilde{W}$ is not a Coxeter group for any simple system $R \subseteq \tilde{W}$ (see also the a bit stronger statement in \cite{BMN24}).

\begin{theorem}\label{Thm:NotCoxeter}
The hyperbolic cover $\tilde{W}$ is not a Coxeter group.
\end{theorem}
\begin{proof}
 Assume that $(\tilde{W},R)$ is a Coxeter system for some simple system $R \subseteq \tilde{W}$.
Let $(\rho, V_t, \langle -\mid - \rangle)$ be the Tits' representation of $(\tilde{W},R)$. 
Then $\tilde{W}$ is generated by a set of reflections $\{r_\alpha \mid \alpha \in \Phi\}$ with respect
to the set of roots $\Phi \subseteq V_t$. Let $z \in Z(\tilde{W})$  be such that $\langle z \rangle \cong \ZZ$.
Then $z$ commutes with every $r_\alpha$, $\alpha \in \Phi$ and therefore stabilizes every subspace
$\RR\alpha $ of $V_t$. As $z$ leaves the form $\langle - \mid - \rangle$ invariant, it is $z(\alpha) = \lambda(\alpha) \alpha$ for some $\lambda(\alpha) \in \{\pm 1\}$ for every $\alpha \in \Phi$. We conclude from $\langle \Phi \rangle = V_t$,
that $z$ is of order $1$ or $2$ in contradiction to the choice of $z$.
\end{proof}

\subsection{The extended Coxeter system}\label{Sec:ExtCox}

Let $(\overline{W}, \overline{S})$ be a crystallographic Coxeter system with a linear representation on $\overline{V}$, which is  equipped with a bilinear form $\langle -\mid -\rangle$. Further, let $\overline{\Phi} \subseteq \overline{V}$ be a root system  for the Coxeter system and $L(\overline{\Phi})$ its root lattice. Then $\overline{W}$ acts on $L(\overline{\Phi})$. The semi-direct product
$$\mathcal{W} := \mathcal{W}(\overline{W}, \overline{S}):= \overline{W} \ltimes L(\overline{\Phi})$$
is the \defn{extended Coxeter group (by its root lattice)} (see \cite{Lek2}).
We write elements $w \in \mathcal{W}$ as pairs 
$w = (\overline{w},\lambda) $, where $\overline{w}$ is in $\overline{W}$ and $\lambda$ in $L(\overline{\Phi})$.

We further assume that  $(\overline{W}, \overline{S})$ is of star type.
Label the vertices of $\overline{\Gamma}$ by $1$ to $n$, so that $t$ is the label of a vertex of maximal valence.
Further, let $\alpha_i$ be the positive root of $\overline{\Phi}$ related to $s_i \in \overline{S}$,
$\overline{ B}: = \{\alpha_1, \ldots ,\alpha_n\}$, and $$s_{t^{\star}}:= (s_t, \alpha_t).$$

Set
$$\mathcal{S}:= \overline{S}\cup \{s_{1^{\star}}\}\subset \mathcal{W},$$
and call $\mathcal{S}$  \defn{simple system for} $\mathcal{W}$, and $(\mathcal{W}, \mathcal{S})$ \defn{ extended Coxeter system of star type}. 

As a consequence of Proposition~\ref{ExactStructure} we obtain the next statement.

\begin{theorem}\label{Prop:CharExtendedWeylNotTubular}
The following holds.
\begin{itemize}
\item[(a)]
Let $(\tilde{W}, \tilde{S})$ be the hyperbolic cover of a tubular elliptic Weyl system, then it  is isomorphic to an extended  Coxeter system  of star type.
\item[(b)]
Let $(\mathcal{W}, \mathcal{S})$ be
an extended Coxeter system of star type, where $\mathcal{W} = \mathcal{W}(\overline{W}, \overline{S})$. 
If $\overline{\Gamma}$ is an affine diagram, then $(\mathcal{W}, \mathcal{S})$ is isomorphic to the hyperbolic cover of a tubular elliptic Weyl system.
\end{itemize}
\end{theorem}
\begin{proof}
Statement (a) is a consequence of Proposition~\ref{ExactStructure}
(a) and (b).
Let $\overline{\Gamma}$ be of affine type. Then according to Proposition~2.18 (c) there 
is an elliptic Weyl system $(W,S)$ such that $(\overline{W}, \overline{S})$ is isomorphic 
to $(W_b,S_b)$. Label $\overline{\Gamma}$ and $\Gamma_b$ identically by the set of labels  $\{0 , \ldots ,n\}$, and let $t$ be the label of the vertex of valency at least $3$.
Then an isomorphism between $(\overline{W}, \overline{S})$ and $(W_b,S_b)$ is given by the 
map $s_i \mapsto s_{i}$. Let $\tilde{W}$ be the hyperbolic cover of $W$. Then $s_i \mapsto \tilde{s}_i$, for $0 \leq i \leq n-1$ defines an
isomorphism from $(W_b,S_b)$ to $(\tilde{W}_b,\tilde{S}_b)$ by Lemma~\ref{lem:IsoAffUG}.
Therefore, the map sending  $s_i$ to $\tilde{s}_{i}$, for $0 \leq i \leq n$, is an isomorphism from $(\overline{W}, \overline{S})$ to $(\tilde{W}_b,\tilde{S}_b)$, which can be
extended to an isomorphism between $(\mathcal{W}, \mathcal{S})$  and $(\tilde{W}, \tilde{S})$ by sending $s_{t^{\star}}$ to $\tilde{s}_{t^{\star}}$ by Proposition~\ref{ExactStructure}. This shows  (b).
\end{proof}

\subsection{A normal form for $\tilde{W}$}
We derive from Proposition~\ref{ExactStructure} a normal form for the elements in $\tilde{W}$.

\begin{corollary}\label{Cor:NormalFormTildeW}
Let $\tilde{w} \in  \tilde{W}$.
\begin{itemize}
\item[(a)] The element $\tilde{w}$ can be  written uniquely as a triple $(w_{\fin}, \TR_b(\lambda), \tilde{\TR}_a(\mu))$, where $w_{\fin} \in W_{\fin},  \lambda \in L(\Phi_{\fin})$ and $\mu \in  L(\Phi_b)$, or in vector notation as {\scriptsize{$ \left( \begin{array}{c} w_{\fin} \\ \lambda \\ \mu \end{array} \right)$}}.
\item[(b)] The element $\tilde{w}$ can also be written uniquely as the tuple $(\tilde{w}_b,\tilde{\TR}_a(\mu))$, where $\tilde{w}_b = w_{\fin} \TR_b(\lambda) \in W_b$, or in vector notation as {\scriptsize{$\left( \begin{array}{c} \tilde{w}_b\\ \mu\end{array} \right)$}}.
\item[(c)] The reflection $\tilde{s}$ in  $\mathrm{End}(\tilde{V})$ with respect to the root $\alpha +l b + ka$, where $\alpha$ is in $\Phi_{\fin}$ and $k,l \in \ZZ$, is in vector notations, respectively
{\scriptsize	{$$ \left( \begin{array}{c} s_\alpha \\ l \alpha \\  k \alpha + kl b \end{array} \right) ~\mbox{and}~ \left( \begin{array}{c} s_{\alpha +l b} \\  k( \alpha + l b)\end{array} \right) .$$}}
\item[(d)] The multiplications of two elements are, respectively, given by 
{\scriptsize {$$ \left( \begin{array}{c} w_1\\ \lambda_1 \\ \mu_1 \end{array} \right)
\left( \begin{array}{c} w_2\\ \lambda_2 \\ \mu_2\end{array} \right)
= \left( \begin{array}{c} 
w_1w_2\\w_2^{-1}(\lambda_1) + \lambda_2 \\ 
t_b(\lambda_2)^{-1}w_2^{-1}(\mu_1) + \mu_2\end{array} \right)
$$}} and 
{\scriptsize $$ \left( \begin{array}{c} \tilde{w}_1\\ \mu_1 \end{array} \right)
\left( \begin{array}{c} \tilde{w}_2\\ \mu_2\end{array} \right)
= \left( \begin{array}{c} 
\tilde{w}_1 \tilde{w}_2\\\tilde{w}_2^{-1}(\mu_1) + \mu_2 \\ 
\end{array} \right).
$$}
\end{itemize}
\end{corollary}

\section{The Coxeter transformation in $\tilde{W}$}\label{CoxeterTransMulti}

In this section, we discuss our main actor, the Coxeter transformation. It is the map
$$\tilde{c} := \tilde{s}_{1} \cdots \widehat{{s}_t } \cdots   \tilde{s}_{n}  \tilde{s}_{0} \tilde{s}_{t}  \tilde{s}_{t^\star} $$
in $\tilde{W}$. Analogously, we define the Coxeter transformation  $\hat{c}$ in  $\hat{W}$.

\begin{remark}\label{Rem:ConjCoexTr}
By the same proof as given for $c$ in $W$ in  \cite[Proposition 5.21 (c) and (a)]{BWY21} the definition of the Coxeter transformation $\tilde{c}$ in $\tilde{W}$ is independent of the elliptic root basis $B$ chosen  and of the order of $\tilde{S}\setminus{\{\tilde{s}_t, \tilde{s}_{t^\star} \}}$ chosen up to conjugacy. 
\end{remark}

Next, we calculate the Coxeter transformation $\tilde{c} \in \tilde{W}$ using the vector notation.
\begin{align*}
\tilde{c} = \tilde{s}_{1} \cdots \widehat{s_t} \cdots   \tilde{s}_{n}  \tilde{s}_{0} \tilde{s}_{t}  \tilde{s}_{t^\star} & = \left( \begin{array}{c} s_1 \\ 0 \\  0  \end{array} \right)  \cdots   \widehat{ \left( \begin{array}{c} s_t \\ 0 \\  0  \end{array} \right) }  \cdots \left( \begin{array}{c} s_{n-2} \\ 0 \\  0  \end{array} \right) 
 \left( \begin{array}{c} s_{\tilde{\alpha}} \\  - \tilde{\alpha}  \\  0  \end{array} \right) 
 \left( \begin{array}{c} s_t \\ 0 \\  0  \end{array} \right)
 \left( \begin{array}{c} s_t \\ 0 \\  \alpha_t  \end{array} \right)\\
&= \left( \begin{array}{c} s_1 \cdots \widehat{s_t} \cdots s_{n-2} \\ 0  \\  0  \end{array} \right) 
 \left( \begin{array}{c} s_{\tilde{\alpha}} \\  - \tilde{\alpha}  \\  0  \end{array} \right) 
 \left( \begin{array}{c} \idop   \\ 0\\ \alpha_t  \end{array} \right)\\
&=   \left( \begin{array}{c} s_1 \cdots \widehat{s_t} \cdots s_{n} s_{\tilde{\alpha}}  \\ 
 	- \tilde{\alpha} \\ \alpha_t \end{array} \right).
\end{align*}

 \subsection{The reflections}\label{Sec:Reflections}
 	
 Let $(W,S)$ be a tubular elliptic  Weyl system with set of reflections  $T$ and  extended root system $\Phi$. Let $(\tilde{W}, \tilde{S})$ and $(\hat{W}, \hat{S})$ be the hyperbolic covers of $(W,S)$. We analogously  to the elliptic group define the set of reflections  $\tilde{T}$ and $\hat{T}$ as
 $$\tilde{T}:= \cup_{w \in \tilde{W}} w^{-1} \tilde{S} w~\mbox{and}~\hat{T}:= \cup_{w \in \hat{W}} w^{-1} \hat{S} w,~\mbox{respectively}.$$
 	
 	As $T$ generates $W$, each $w \in W$ is a product $w = t_1 \cdots t_r$ where $t_i \in T$. If this product is of minimal length, then we say that this $T$-factorization is \defn{reduced} and that $\ell_T(w) =r$. For $w  \in W$ with $\ell_T(w) =r$ we put 
 	$$
 	\Red_T(w) := \{ (t_1 , \ldots, t_r) \in T^r \mid w=t_1 \cdots t_r \}.
 	$$
 	Analogously, we define $\ell_{\tilde{T}}(\tilde{w})$ and $\ell_{\hat{T}}(\hat{w})$, as well as $\mbox{Red}_{\tilde{T}}(\tilde{w})$ and $\mbox{Red}_{\hat{T}}(\hat{w})$, where $\tilde{w} \in \tilde{W}$ and $\hat{w} \in \hat{W}$.

 	\begin{Lemma}\label{Lem:ReflectionsW-tildeW}
 	The following holds.
 	\begin{itemize}
 	\item[(a)]	Every reflection $t \in T \subset W$ has a unique preimage in $\tilde{T}\subset \tilde{W}$ and in $\hat{T} \subseteq \hat{W}$, which we denote by $\tilde{t}$ and $\hat{t}$, respectively. Moreover, $\varphi(\tilde{T}) = T$.
 	\item[(b)] Let $t_1, t_2, t_3 \in T$. It is $t_1^{t_2} = t_3 $ in $W$ if and only if $\tilde{t} _1^{\phantom{.}\tilde{t}_2} =  \tilde{t}_3 $ in $\tilde{W}$.
    \item[(c)] If $(t_1, \ldots, t_k)$ and $(u_1, \ldots, u_k)$ are in $T^k$ for some $k \in \NN$, and if there is a braid $\sigma \in {\mathcal{B}}_{k} $ that maps $(t_1, \ldots, t_k)$ onto $(u_1, \ldots, u_k)$, then $\sigma(\tilde{t}_1, \ldots, \tilde{t}_k) = (\tilde{u}_1, \ldots, \tilde{u}_k)$.
    \end{itemize}
\end{Lemma}
\begin{proof}
Statement (a) follows directly from  Corollary~\ref{Cor:NormalFormTildeW}, and (b)  from Proposition~\ref{StructureTildeW} and (a), while (c) is a consequence of (b).
\end{proof}

In the next section, we determine $\ell_T(x)$ for $x \in \{c,\tilde{c}, \hat{c}\}$. Therefore, the next lemma will be of great help.
 	
\begin{Lemma}\label{Lem:lengthtildeW-W}
Set $z:= \tilde{E}(a \otimes b) \in Z(\tilde{W})$. If $\ell_{\tilde{T}}(\tilde{c}z^i) = \ell_{\tilde{T}}(\tilde{c})$ for all $i \in \ZZ$, then $\ell_{\tilde{T}}(\tilde{c}) = \ell_T(c)$, and $\cup_{i \in \ZZ} 	\Red_{\tilde{T}}(\tilde{c}z^i)$ are bijectively mapped  to $\Red_T(c)$, and each Hurwitz-orbit is mapped onto a  Hurwitz-orbit.
\end{Lemma}
\begin{proof} 	
Let $c = t_1 \cdots t_m$ be a $T$-factorization. Then $\tilde{c}z^i = \tilde{t}_1 \cdots \tilde{t}_m$ for some $i \in \Z$. Thus 
$\ell_T(c) \geq \ell_{\tilde{T}}(\tilde{c}z^i) = \ell_{\tilde{T}}(\tilde{c})$. On the other hand $\tilde{c} = \tilde{t}_1 \cdots \tilde{t}_m$, where $\tilde{t}_1, \ldots , \tilde{t}_m \in \tilde{T} $, implies that $c = t_1 \cdots t_m$ is a $T$-factorization, and 
$\ell_{\tilde{T}}(\tilde{c}) \geq \ell_T(c)$. This shows $\ell_{\tilde{T}}(\tilde{c}) = \ell_T(c)$. Therefore, $\varphi$ induces a map from 
$\cup_{i \in \ZZ} 	\Red_{\tilde{T}}(\tilde{c}z^i)$ to $\Red_T(c)$.
We conclude with
Lemma~\ref{Lem:ReflectionsW-tildeW} (a) that this map is a bijection, and the last statement follows from 
Lemma~\ref{Lem:ReflectionsW-tildeW} (c).
\end{proof}		
 	
\subsection{The reflection length of a Coxeter transformation}\label{sec:reflength}
The aim of this subsection is to show Proposition~\ref{Lem:redCoxElt}. More precisely, we will determine the reflection length of a Coxeter transformation $\tilde{c}$ of the extended Weyl group $\tilde{W}$, which then also implies $\ell_T(c)$.
 	 
We will show that $\ell_{\tilde{T}}(\tilde{c}) = |\tilde{S}| = n+2$, which shows that the factorization of  $\tilde{c}$  in pairwise different simple reflections is $\tilde{T}$-reduced. This will then be used in the proof of Theorem~\ref{conj:WeightProjElliptic}.
 	
Since $\hat{W}$  and $\tilde{W}$ are isomorphic groups via an isomorphism $\psi$, which maps $\hat{s}_i $ onto $ \tilde{s}_i $, for $0 \leq i \leq n$, and $\hat{s}_{t^\star}$ onto $\tilde{s}_{t^{\star}} $ and since $\psi$ maps $\hat{T} $ onto $\tilde{T}$,  it is $\ell_{\hat{T}}(\hat{c}) = \ell_{\tilde{T}}(\tilde{c})$.
 	
The reflection length of $\hat{c}$ can be proven using Scherk's Theorem as we show now. In order to formulate Scherk's theorem, we recall some notation. Let $k$ be a field of characteristic different from $2$, and  $U$ a finite-dimensional $k$-vector space that is equipped with a symmetric bilinear form $(- \mid -)$. A null-space (totally isotropic subspace) of\linebreak $(U,(- \mid-))$ is a subspace that consists only of isotropic vectors, that is, as char$(k) \neq 2$, a subspace where $(- \mid -)$ vanishes.
 	
 \begin{theorem}[{Scherk, \cite[Theorem 260.1]{ST89}}]\label{thm:scherk}
 Let $k$ be a field of characteristic different from $2$, $U$ a finite-dimensional $k$-vector space with a non-degenerate symmetric bilinear form $(-\mid -)$. Further let $\sigma \neq \idop_{U}$ be an isometry of $U$. Let $F= C_U(\sigma)$ be the fixed space of $\sigma$ in $U$, and $n$  the dimension of $F^{\perp}$,  the orthogonal complement of $F$ in $U$. If  $F^{\perp}$ is not a null-space,  then $\sigma$ is the product of $n$ and not less than $n$ reflections. If $F^{\perp}$ is a null space, then $\sigma$ is the product of $n+2$ and not less than $n+2$ reflections.
 \end{theorem}
 	
 In order to apply Scherk's Theorem to our situation, we need the well-known fact described in the following lemma. Let $U$ be a finite-dimensional real vector space and attach to it a symmetric bilinear form $(- \mid -)$.
 	
 \begin{Lemma}\label{lem:orth_complement_lin_ind_roots}
 Let $\beta_{1},\ldots,\beta_{m}\in U$ be linear independent and non-isotropic vectors in $U$, and let $u\in U$, then $s_{\beta_{1}}\cdots s_{\beta_{m}}(u)=u$ if and only if $s_{\beta_{i}}(u)=u$ for $1\leq i \leq m$.
 \end{Lemma}
 \begin{proof}
 If $s_{\beta_{1}}\cdots s_{\beta_{m}}(u)=u$, then 
 \[-\frac{(\beta_{1} \mid u)}{(\beta_{1},\beta_{1})}\beta_{1}=s_{\beta_{2}}\cdots s_{\beta_{m}}(u)-u \in \spanr(\beta_{2},\ldots,\beta_{m})\]
 and the linear independence of $\lbrace \beta_{1},\ldots,\beta_{m}\rbrace$ implies $(\beta_{1} \mid u)=0$. The latter is equivalent to $s_{\beta_{1}}(u)=u$. In the same manner, we inductively get $s_{\beta_{i}}(u)=u$ for $2\leq i \leq m$.
\end{proof}
 	
Now let us consider the hyperbolic cover of an elliptic Weyl group $\hat{W}$, the $\RR$-space $\hat{V}$ and the symmetric form $(-\mid -)$ extended to $\hat{V}$ (see Section~\ref{Sec:HyperbolicExtended}), and let $\Phi$ be the elliptic root system attached to $W$.
 	
\begin{Lemma}\label{lem:fix_c}
 Let $\hat{W}$ be a hyperbolic cover of an elliptic Weyl group $W$ and $\hat{c} \in \hat{W}$ be a Coxeter transformation. Then $C_{\hat{V}}(\hat{c}) = R$. 
\end{Lemma}
\begin{proof}
Due to Lemma~\ref{lem:orth_complement_lin_ind_roots} we get $\displaystyle C_{\hat{V}}(\hat{c})=\bigcap_{\alpha\in B}C_{\hat{V}}(\hat{s}_{\alpha})=\bigcap_{\alpha\in B}\alpha^{\perp} = \bigcap_{\alpha\in B_{\fin}}\alpha^{\perp} \cap a^\perp 
\cap b^\perp = R$. 	
The last equality is valid by the construction in the proof of Lemma~\ref{HyperbolicExtension}.
\end{proof}
 	
\noindent	
{\bf Proof of Proposition~\ref{Lem:redCoxElt}}
Let $\hat{c}$ be a Coxeter transformation and $C_{\hat{V}}(\hat{c})$ its fixed space.  By Lemma \ref{lem:fix_c} it holds $ C_{\hat{V}}(\hat{c})^{\perp} = R^\perp = V$, and therefore $\dim (C_{\hat{V}}(\hat{c})^{\perp})= \dim (V) = n+2$. Since $(-\mid-)$ does not vanish on $V$,  Scherk's Theorem \ref{thm:scherk} yields $\ell_{\hat{T}}(\hat{c}) \geq n+2$. As by definition  $\hat{c}$ is the product of $|S| = n+2$ reflections,  it is $\ell_{\hat{T}} (\hat{c}) = n+2$.
 	
It remains to show $\ell_{\tilde{T}}(\tilde{c}) = \ell_T(c)$. We check the condition given in Lemma~\ref{Lem:lengthtildeW-W}. Let $\hat{z}$ be the preimage of  $z= \tilde{E}(a \otimes b)$ in $\hat{W}$, and let $j \in \ZZ$. Then $\hat{z}$ centralizes an $(n+3)$-dimensional subspace of $\hat{V}$, and it is $\hat{z}^j(b') = b' - ja$. By direct calculation, we see $C_{\hat{V}}(\hat{c}\hat{z}^j) = R$. Therefore, the same argument as in the last paragraph gives us $\ell_{\hat{T}} (\hat{c}\hat{z}^j) = n+2$ for each $j \in \ZZ$.  We conclude that $\ell_{\tilde{T}} (\tilde{c}z^j) = n+2$, and the second assertion follows with Lemma~\ref{Lem:lengthtildeW-W}. \hfill \qed

 \subsection{The posets $[\idop,w ]$ for $w  \in W $ and for $w$  in $\tilde{W}$}\label{Sec:IntervalPosets}
 
 Here we introduce the posets for $w \in W$ in $W$ and for $w \in \tilde{W}$ in $\tilde{W}$, which will be used in Section~\ref{Sec:Application}. 
 
 \begin{Definition} \label{def:PrefixPoset}
 
 \begin{itemize}
 \item[(a)] Define a partial order on $W$ by $x \leq y ~ \mbox{if and only if} ~\ell_T(y)= \ell_T(x) + \ell_T(x^{-1}y)$ for $x,y \in W$, called \defn{absolute order}, where $\ell_T$ is the length function on $W$ with respect to $T$.
 \item[(b)] For $w \in W$ the interval $[\idop, w] = \{x \in W \mid \idop \leq x \leq w \}$ is called the \defn{poset} of $w$ with respect to the partial order $\leq$.
 \item[(c)] Let $x = t_1\cdots  t_i$ and $(t_1,\ldots,t_m) \in \Red_T(c)$. We call $(t_1,\ldots , t_i)$ a \defn{prefix} of the factorization $(t_1,\ldots,t_m)$ for $1 \leq i \leq m $, where $m = \ell_T(c)$.\\
\end{itemize}
\end{Definition}
 
Note that $x \leq w$ if and only if there exists a reduced factorization $\left(t_1,\ldots , t_{\ell_T(x)} \right) \in \Red_T(x)$ that  is a prefix of a reduced factorization of $w$. By abuse of notation, we sometimes call $x$ a prefix of $w$. Analogously, we define the same notation for $w, x,y \in \tilde{W}$ and $t_i \in \tilde{T}$ as for $w,x,y \in W$ and $t_i \in T$.

\section{Hurwitz action in the tubular elliptic Weyl group and in its cover}\label{sec:HurwitzElliptic}

In this section we prove Theorem~\ref{thm:MainElliptic} by applying Theorem~\ref{Thm:11in 4}.
Let $(\tilde{W}, \tilde{S})$ be the hyperbolic cover of the elliptic Weyl system $(W,S)$. Then there is an extended Coxeter system $(\mathcal{W},
\mathcal{S})$ that is isomorphic to $(\tilde{W}, \tilde{S})$ by Theorem~\ref{Prop:CharExtendedWeylNotTubular}.  Choose the notation for
$\tilde{S}$ and $\mathcal{S}$ as in Section~\ref{Sec:ExtCox} such that 
$\tilde{s}_i \mapsto s_i$  and $\tilde{s}_{t^{\star}} \mapsto s_{t^{\star}}$
define an isomorphism from $(\tilde{W}, \tilde{S})$ to $(\mathcal{W},\mathcal{S})$.

We set $$c := s_1 \cdots \widehat{s}_t \cdots s_n s_0 s_t s_{t^\star} $$
and 
$$\mathcal{T}:= \cup_{w \in \mathcal{W}} w \mathcal{S} w^{-1}~\mbox{and}~ R_{\mathcal{T}}(c):= \{(t_1, \ldots, t_{n+2}) \in \mathcal{T}^{n+2}~\mid c = t_1 \cdots t_{n+2} \}.$$

\noindent
{\bf Proof of Theorem~\ref{thm:MainElliptic}}
By our setting, it is $\varphi(\tilde{c}) = c$ and  $\varphi(\tilde{T}) =\mathcal{T} $.
In addition, $\varphi$ induces a bijection  from  $\Red_{\tilde{T}}(\tilde{c})$
to $R_{\mathcal{T}}(c)$, which respects the Hurwitz action.
Therefore, the statement is a consequence of  Theorem~\ref{Thm:11in 4}.
 \hfill\qed

\bigskip
\noindent
  In \cite{BWY21} the authors and Yahiatene defined that a reflection factorization $c = t_1\cdots t_{n+2}$ of a Coxeter transformation $c$ in the elliptic Weyl group $W$ is generating if $W = \langle t_1, \ldots, t_{n+2} \rangle$, and denoted by $\Redd(c)$ the set of reduced generating reflection factorizations. The three authors showed that the braid group $\mathcal{ B}_{n+2}$  acts transitively on  $\Redd(c)$ in types $E_n^{(1,1)}$ ($n = 6,7,8$) (see \cite[Theorem~1.3]{BWY21}) and with at most two orbits in type $D_4^{(1,1)}$ (see the proof of \cite[Theorem~1.3]{BWY21}). Notice that in a previous version of \cite{BWY21} the two possible orbits in type $D_4^{(1,1)}$ were overlooked, which was observed by Charly Schwabe. Here, we clarify the situation. Representatives of the two possible orbits are $(s_1,s_3,s_4, s_0, s_2, s_{2^\star})$ and $(s_{\alpha_1-a},s_{\alpha_3-a},s_{\alpha_4-a}, s_{\widetilde{\alpha} +a -b}, s_2, s_{{\alpha_2} -a})$ (see Equations (7) and (8) of the proof of \cite[Theorem~1.3]{BWY21}).

\begin{corollary}\label{Cor:HurwitzW-tildeW}
If the elliptic root system $\Phi$ is of type $E_6^{(1,1)}, E_7^{(1,1)}$ or $E_8^{(1,1)}$, then $\varphi$ maps $\Red_{\tilde{T}}(\tilde{c})$ bijectively onto $\Redd(c)$, and in type $D_4^{(1,1)}$ there are two Hurwitz orbits on  $\Redd(c)$, and the set $\Red_{\tilde{T}}(\tilde{c})$ is mapped bijectively onto the Hurwitz orbit of $\Redd(c)$ that contains $(s_1,s_3,s_4, s_0, s_2, s_{2^\star})$.
\end{corollary}

\begin{proof} 
 Let $(\tilde{t}_1, \ldots , \tilde{t}_{n+2}), (\tilde{u}_1, \ldots, \tilde{u}_{n+2})  \in \Red_{\tilde{T}}(\tilde{c})$. Then $(t_1, \ldots , t_{n+2}), (u_1, \ldots, u_{n+2}) $ are in $\Red_T(c)$ and  $(\hat{t}_1, \ldots , \hat{t}_{n+2}), (\hat{u}_1, \ldots, \hat{u}_{n+2}) $ in $\Red_{\hat{T}}(\hat{c})$. Moreover, if $\sigma $ is an element in the braid group $\mathcal{B}_{n+2}$, then $\sigma$ maps $(t_1, \ldots , t_{n+2})$ onto $(u_1, \ldots, u_{n+2}) $ if and only if $\sigma$ maps  $(\tilde{t}_1, \ldots, \tilde{t}_{n+2})$ onto $(\tilde{u}_1, \ldots, \tilde{u}_{n+2})$ by Lemma~\ref{Lem:ReflectionsW-tildeW} (c). Hence $\varphi$ maps the Hurwitz orbit Red$_{\tilde{T}}(\tilde{c})$ bijectively onto a Hurwitz orbit of $c$. The assertion follows if $\Phi$ is of type $E_n^{(1,1)}$ with  $n \in \{6,7,8\}$ by \cite[Theorem~1.3]{BWY21}, as the image of $\tilde{c} = \tilde{s}_1 \cdots \tilde{s}_{t-1} \tilde{s}_{t+1} \cdots \tilde{s}_n \tilde{s}_0  \tilde{s}_t \tilde{s}_{t^\star}$  under $\varphi$ is a generating factorization.
 
If $\Phi$ is of type $D_4^{(1,1)}$, then the calculation in $\tilde{W}$
\begin{align*} \tilde{s}_{\alpha_1-a} \tilde{s}_{\alpha_3-a} \tilde{s}_{\alpha_4-a} \tilde{s}_{\widetilde{\alpha} +a -b} \tilde{s}_2 \tilde{s}_{{\alpha_2} -a} = 
\begin{bmatrix} s_{1} \\0 \\ -\alpha_1 \\ \end{bmatrix} \cdot 
\begin{bmatrix} s_{3} \\0\\ -\alpha_3 \end{bmatrix} \cdot 
\begin{bmatrix} s_{4} \\0 \\-\alpha_4 \end{bmatrix} \cdot 
\begin{bmatrix} s_{\widetilde{\alpha}} \\ -\widetilde{\alpha} \\ \widetilde{\alpha} -b\end{bmatrix} \cdot 
\begin{bmatrix} s_{2} \\ 0 \\ 0\end{bmatrix} \cdot 
\begin{bmatrix} s_{2} \\ 0\\-\alpha_2 \end{bmatrix} &= 
\begin{bmatrix} s_1s_3s_4s_{\widetilde{\alpha}} \\ -\widetilde{\alpha} \\ 
\alpha_2 -b
\end{bmatrix}\\ &\neq \begin{bmatrix} s_1s_3s_4s_{\widetilde{\alpha}} \\ -\widetilde{\alpha} \\ 
\alpha_2 
\end{bmatrix} = \tilde{c}
\end{align*}
yields that there are two Hurwitz orbits on  $\Redd(c)$, and that $\Red_{\tilde{T}}(\tilde{c})$ is mapped  bijectively onto the Hurwitz orbit of  $\Redd(c)$, which contains $(s_1,s_3,s_4, s_0, s_2, s_{2^\star})$, as the product of the preimages of the reflections equals $\tilde{c}$.
\end{proof}

\section{An application}\label{Sec:Application}

In this section, we apply our results to the category $\COH(\XX)$ by proving Theorem \ref{conj:WeightProjElliptic} and thus improving Theorem~1.2, one of the main results of \cite{BWY21}. All the notation concerning $\COH(\XX)$ is explained in \cite{BWY21}. We also recall that $\tilde{W}$ is the hyperbolic cover of the tubular extended Weyl group $W$, and that $\tilde{c}$ is the Coxeter transformation in $\tilde{W}$ related to the Coxeter transformation $c$ in $W$ (see Section~\ref{CoxeterTransMulti}).

\begin{proof}[\textbf{Proof of Theorem \ref{conj:WeightProjElliptic}}]
We adapt the proof of \cite[Theorem~1.2]{BWY21} to our situation. Let $\XX$ be of tubular type. Then the extended Weyl group $W:= W(\XX)$ attached to $\XX$ is an elliptic Weyl group as shown in \cite{BWY21}.

Due to Corollary~\ref{Cor:HurwitzW-tildeW} we can consider the Hurwitz orbit Red$_{\tilde{T}}(\tilde{c})$ instead of the orbit $\Redd(c)$ in the types $E_n^{ (1,1)}$ (for $n = 6,7,8$) and instead of one of the two orbits in type $D_4^{(1,1)}$, and the arguments given in the proof of Theorem~1.2 in \cite{BWY21} work also for the hyperbolic cover $\tilde{W}$ of $W$. In the following we point out the differences. The major difference between the two proofs is that we do not need condition (b) of Theorem~1.2 anymore. Every reflection factorization of $\tilde{c}$ in $\tilde{W}$ automatically generates $\tilde{W}$.
 
Let $(E_1, \ldots , E_{n+2})$ be a complete exceptional sequence in $\COH(\XX)$. By \cite[Lemma~3.9, Corollary~3.10]{BWY21}  we have $c=s_{[E_1]} \cdots s_{[E_{n+2}]}$.

For $U=\Thi(E_1, \ldots , E_r)$, where $r \leq n+2$,  put $\cox(U):=\tilde{s}_{[E_1]} \cdots \tilde{s}_{[E_r]} \in \tilde{W}$. This map is well defined  by exactly the same  proof as given in \cite[Section~4.3]{BWY21}. 
 
The injectivity also follows analogously; it should just be noted that we do not need condition (b) anymore, as every factorization in Red$_{\tilde{T}}(\tilde{c})$  generates $\tilde{W}$. The map $\cox(-)$ is surjective by Corollary~4.9 of \cite{BWY21} and Corollary~\ref{Cor:HurwitzW-tildeW}. Finally $\cox(-)$ is order preserving by exactly the same argument as given in the last paragraph of the proof of Theorem~1.2 in \cite{BWY21}.
\end{proof}

\section{Appendix: The hyperbolic cover of a Coxeter system}

In this section, we introduce the concept of a hyperbolic cover $(\tilde{W}, \tilde{S})$ of a Coxeter system $(W,S)$, and prove that we do not obtain a new system, i.e. $\tilde{W} \cong W$. We apply this result to show that neither an elliptic Weyl group nor its hyperbolic cover are Coxeter groups.

Let $(W,S)$ be a Coxeter system of rank $n$ with Tits' representation $(\rho, V, (-\mid -))$, where $V$ is an $\RR$-space, $\rho$ a linear representation $\rho: W \rightarrow \GL(V)$ and $(-\mid -)$ the related symmetric bilinear form. Let $(m_1,m_2,m_0)$ be the signature of $(-\mid -)$.

Furthermore, let $\Phi$ be the root system of the Tits' representation and $B = \{\alpha_1 , \ldots, \alpha_n\}$ be a fundamental system for $\Phi$ such that $S = \{s_1, \ldots, s_n\}$ and $\rho(s_i)$ is the reflection of $V$ with respect to $\alpha_i$. In the following, we will identify $s_i$ and $\rho(s_i)$.

\subsection{The hyperbolic extended space  of a Coxeter system}

A \defn{hyperbolic extended space} of $ (\rho, V,(-\mid -)) $ (or simply of $(V, (- \mid -)$) is an $\RR$-space $(\tilde{V}, (-\mid -))$ of signature $(\tilde{m}_1, \tilde{m}_2, \tilde{m}_0)$ that extends $(V, (-\mid -))$ so that there is $x \in \NN_0$ with $\tilde{m}_i = m_i +x$ for $i = 1,2$ and $\tilde{m}_0 = m_0 -x$.

\begin{Lemma}
A hyperbolic cover of $(V, (-\mid -))$ exists for $0 \leq x \leq m_0$.  
\end{Lemma}
\begin{proof}
This follows with the analogous proof given in Section~\ref{Sec:HyperbolicExtended}.    
\end{proof}

\begin{remark}\label{Rem:HyperCoverRad0}
Notice that the (unique) hyperbolic extended space of $(V,(-\mid -))$ where the form is non-degenerate is the space $(V,(-\mid -))$ itself (here $x= 0 $).
\end{remark}

Let $X$ be a basis of $V$ such that the Gram matrix of the symmetric bilinear form $(-\mid -)$ has diagonal form with respect to $X$, and let $V_+$ and $V_-$ be the $\RR$-span of the vectors of $v \in X$ such that $(v \mid v) > 0$ and $(v \mid v) < 0$, respectively. Further, let $X_0:= \{v \in X \mid (v\mid v) = 0\}$, which then is a basis for the radical $R$ of $V$. 

Let $(\tilde{V}, (- \mid -))$ be a hyperbolic cover of $(V, (-\mid -))$ of signature $(\tilde{m}_1,\tilde{m}_2,\tilde{m}_0)$ with $x:= m_0 - \tilde{m}_0$, and let $\tilde{R} \subseteq R$ be the radical of the form $(-\mid -)$ in $\tilde{V}$. Notice that we can (and do) choose $X$ so that it contains a basis $\tilde{X}_0$ of $\tilde{R}$. Let $X_0\setminus{\tilde{X}_0} = \{b_1, \ldots , b_x\}$.

\begin{Lemma}\label{Lem:StructureHypExtension}
Let $(\tilde{V}, (- \mid -))$ be a hyperbolic extension of $(V, (-\mid -))$ of signature $(\tilde{m}_1,\tilde{m}_2,\tilde{m}_0)$ with $x:= m_0 - \tilde{m}_0$. Then there exist vectors $b_1', \ldots , b_x' \in \tilde{V}$ such that $(b_i, b_i')$ forms a hyperbolic pair for $1 \leq i \leq x$, and such that $\tilde{V}$ decomposes as follows
$$\tilde{V} = V_+ \oplus V_- \oplus \spanr(b_1, b_1') \oplus \cdots  \oplus \spanr(b_x, b_x') \oplus \spanr(\tilde{X}_0).$$
\end{Lemma}
\begin{proof}
We continue to use the notation introduced before the lemma. Set 
$$V_1:= V_+ \oplus V_- \oplus \cdots\oplus \spanr(X_0\setminus{\{b_1\}})) \subset V \subset \tilde{V} .$$ 
In the non-degenerate  space $\tilde{V}/\tilde{R}$ we calculate that $V_1^\perp$ and $V^\perp$ are of dimensions $x +1$ and $x$, respectively. This shows that there is a vector $b_1' \in \tilde{V}$ in $V_1^\perp$ such that $(b_1\mid b_1') = 1$. Then, by an argument similar to that in the proof of Lemma~\ref{HyperbolicExtension} we may assume that $b_1'$ is an isotropic vector and $(b_1, b_1')$ a hyperbolic pair. Next set 
$$V_2:= V_+ \oplus V_- \oplus \spanr(b_1, b_1') \oplus \spanr(X_0\setminus{\{b_1, b_2\}})) \subseteq \tilde{V}.$$ 
Then $V_2$ is of dimension $\dim(V)$ and therefore $V_2^\perp$ and $(V_2\oplus \RR b_2)^\perp$ are in $\tilde{V}$ of dimension $x$ and $x-1$, respectively. Hence, we find a vector $b_2' \in V_2^\perp$ such that $(b_2, b_2')$ is a hyperbolic pair. We continue this procedure by considering the subspaces $V_i$ for $1 \leq i \leq x$, which are of dimension $\dim(V)+ i -2$, and by constructing the vectors $b_3', \ldots ,b_x'$. As the dimension of the right vector space in the equality shown in the lemma statement is $\dim(V) + x$, the equality claimed is true.
\end{proof}

\subsection{The hyperbolic cover of a Coxeter system}

The definition of a hyperbolic cover of a Coxeter system is analogous to that of a hyperbolic cover of an elliptic Weyl group.

Let $(\tilde{V}, (-\mid -))$ be a hyperbolic extended space of $(V, (-\mid -))$. Then we have $\Phi \subset V \subset \tilde{V}$ and we can consider the reflection $s_i$ of $V$ with respect to $\alpha_i \in B$ as a reflection $\tilde{s}_i$ of $\tilde{V}$ with respect to $\alpha_i$. We set 
$$ \tilde{S}  = \{\tilde{s}_i \mid 1 \leq i \leq n \}~\mbox{and}~\tilde{W} :=  \langle \tilde{S} \rangle,$$
and call $(\tilde{W}, \tilde{S})$ a \defn{hyperbolic cover of the Coxeter system} $(W,S)$. If the sets $S$ and $\tilde{S}$ are clear from the context, we will also simply speak of the hyperbolic cover $\tilde{W}$ of $W$.

\begin{Proposition}\label{Thm:HyperCoverCoxeter}
Every hyperbolic cover of a Coxter system is isomorphic to the Coxeter system.
\end{Proposition}
\begin{proof}
Let $(\tilde{W}, \tilde{S})$ be a hyperbolic cover of a Coxeter system $(W,S)$ with respect to the Tits' representation $(\rho,V, (-\mid -))$ of $(W,S)$. 

Let $S = \{s_1, \ldots, s_n\}$. The definition of the reflections $\tilde{s}_i$ yields that the map $ \tilde{S} \rightarrow S$, $\tilde{s}_i \mapsto (\tilde{s}_i)_{\mid V} = s_i$ induces an epimorphism $\varphi$ from $\tilde{W}$ to $W$. The proof strategy is to show that the map $\psi$ sending $s_i$ to $\tilde{s}_i$ defines a homomorphism $\psi$ from $W$ to $\tilde{W}$.

As $(W,S)$ is a Coxeter system, there exist $m_{ij} \in \NN \cup \{\infty\}$, for $1 \leq i,j \leq n$, with $m_{ii}= 1$ and $m_{ij}= m_{ji}$ such that 
$$W = \langle S \mid (s_is_j)^{m_{ij}} = 1, 1 \leq i,j \leq n\rangle.$$ 
We check that the elements in $\tilde{S}$ satisfy these relations.

As $\tilde{s}_i$ is a reflection of $\tilde{V}$, we get $\tilde{s}_i^2 = 1$, for $1 \leq i \leq n$. It remains to show that $(\tilde{s}_i\tilde{s}_j)^{m_{ij}} =1$, for $i \neq j$. If $m_{ij} = \infty$, then $(\tilde{s}_i\tilde{s}_j)_{\mid V} = s_is_j$ is of infinite order and therefore $\tilde{s}_i\tilde{s}_j$ is also of infinite order.

So, assume that $m_{ij} < \infty$. Let $b_1', \ldots ,b_x' \in \tilde{V}$ be vectors as stated in Lemma~\ref{Lem:StructureHypExtension}. As $(\tilde{s}_i\tilde{s}_j)_{\mid V} = s_is_j$ we need to show that $(\tilde{s}_i\tilde{s}_j)^{m_{ij}}(b_k') = b_k'$ is true for $1 \leq k \leq x$. Set $d:= s_is_j$ and $\tilde{d}:= \tilde{s}_i\tilde{s}_j$. And let $\alpha_i,\alpha_j \in \Phi$ be roots related to $s_i$ and $s_j$, respectively. Then $\tilde{d}(b_k') = b_k' +u$ for some $u$ in the plane $P_{ij}:= \spanr(\alpha_i, \alpha_j)$ by the definition of the reflections $\tilde{s}_i$ and $\tilde{s}_j$. We conclude
$$\tilde{d}^2(b_k') = \tilde{d}(b_k') + \tilde{d}(u) = b_k' + u + \tilde{d}(u)$$
and 
$$\tilde{d}^{m_{ij}}(b_k') = b_k' + u + \tilde{d}(u) + \cdots + \tilde{d}^{m_{ij}-1}(u).$$ 
The linear map $\tilde{d}$ acts on $P_{ij}$ as the rotation $d$. In particular, $\tilde{d}$ fixes only the $0$-vector in $P_{ij}$. Thus, as
$$d(u+ d(u) + \cdots + d^{m_{ij}-1}(u)) = u+ d(u) + \cdots + d^{m_{ij}-1}(u), $$ 
we get 
$$0 =u+ d(u) + \cdots + d^{m_{ij}-1}(u) = u+ \tilde{d}(u) + \cdots + \tilde{d}^{m_{ij}-1}(u),$$ 
and $\tilde{d}^{m_{ij}}(b_k') = b_k'$.

Hence $\psi$ defines a homomorphism from $W$ to $\tilde{W}$, and we get  $\psi \circ \varphi = id_{\tilde{W}}$ and $\varphi\circ \psi = id_W$, which shows that $(\tilde{W}, \tilde{S} )$ and $(W,S)$ are isomorphic.
\end{proof}

\bibliography{mybibDerived2}
\bibliographystyle{amsplain}
\end{document}